\documentclass[12pt]{amsart}
\usepackage{graphicx}
\usepackage{amsmath}
\usepackage{amssymb}
\usepackage{setspace}
\usepackage{amsthm}
\usepackage{color}

\usepackage[hide]{ed} 

\newtheorem{thm}{Theorem}[section]
\newtheorem{cor}[thm]{Corollary}

\newtheorem{lem}[thm]{Lemma}
\newtheorem{prop}[thm]{Proposition}
\newtheorem{claim}[thm]{Claim}
\theoremstyle{definition}

\theoremstyle{remark}
\newtheorem{rem}[thm]{Remark}

\numberwithin{equation}{section}

\newcommand{\Real}{\mathbb R}
\newcommand{\NN}{\mathbb N}
\newcommand{\ZZ}{\mathbb Z}
\newcommand{\NNN}{\mathcal N}

\newcommand{\RR}{\mathbb R}
\newcommand{\TTT}{\mathcal{T}}
\newcommand{\eps}{\varepsilon}

\newcommand{\tu}{\tilde{u}}

\newcommand{\ent}[1]{\left\lfloor #1 \right\rfloor} 

\DeclareMathOperator{\Image}{\mathcal{I}}

\begin{document}
\title[Self-similar solutions of supercritical wave equations]{Self-similar solutions of energy-supercritical focusing wave equations in all dimensions}

\author{Wei Dai, Thomas Duyckaerts}

\address{School of Mathematical Sciences, Beihang University (BUAA), Beijing 100083, P. R. China, and LAGA, UMR 7539, Institut Galil\'{e}e, Universit\'{e} Sorbonne Paris Cit\'e, 93430 - Villetaneuse, France}
\email{weidai@buaa.edu.cn}

\address{Institut Universitaire de France, and LAGA, UMR 7539, Institut Galil\'{e}e, Universit\'{e} Sorbonne Paris Nord, 93430 - Villetaneuse, France}

\thanks{W. Dai is supported by the NNSF of China (No. 11971049), the Fundamental Research Funds for the Central Universities and the State Scholarship Fund of China (No. 201806025011). T. Duyckaerts is supported by the Institut Universitaire de France and partially supported by the Labex MME-DII.}

\begin{abstract}
In this paper, we prove the existence of a countable family of regular spherically symmetric self-similar solutions to focusing energy super-critical semi-linear wave equations
\begin{equation*}
  \partial_{tt}u-\Delta u=|u|^{p-1}u \qquad \text{in} \,\, \mathbb{R}^{N},
\end{equation*}
where $N\geq 3$, $1+\frac{4}{N-2}<p$, and, if $N\geq 4$, $p \leq 1+\frac{4}{N-3}$. This was previously known only in the case $N=3$, for integer $p$ (see Bizo\'{n}, Maison and Wasserman \cite{BMW}). We also study the asymptotics of these solutions.
\end{abstract}
\maketitle {\small {\bf Keywords:} Semi-linear wave equations, Self-similar solutions, Focusing, Energy super-critical. }

\section{Introduction}
Consider the semi-linear wave equation on $\RR^N$:
\begin{equation}\label{0-0}
\partial_{tt}\Phi-\Delta\Phi-|\Phi|^{p-1}\Phi=0, \quad \Phi=\Phi(t, x), \quad x \in \mathbb{R}^{N},
\end{equation}
where $N\geq 1$ is an integer, and $p>1$.
The equation is invariant under the scaling
\begin{equation}\label{0-2}
\Phi(t, x) \mapsto \Phi_{\lambda}(t, x)=\frac{1}{\lambda^{\alpha}}\Phi\left(\frac{t}{\lambda},\frac{x}{\lambda}\right), \quad \alpha:=\frac{2}{p-1}.
\end{equation}
The scale invariant Sobolev space is $\dot{H}^{s_c}\times \dot{H}^{s_c-1}$, with $s_c=\frac{N}{2}-\alpha$. Our main results will concern the energy-supercritical case $s_c>1$, which corresponds to $N\geq 3$, $p>\frac{N+2}{N-2}$.

The nonlinear term has a focusing sign, which means that it tends to magnify the amplitude of the wave. If $\Phi$ is small, this term is negligible and the evolution is essentially linear, leading to dispersion. However, if $\Phi$ is large, the dispersive effect of the Laplacian may be overcome by the focusing effect of the nonlinearity and a singularity may form.  Such a phenomenon, usually referred to as blow-up, has been intensively studied since the pioneering works by Keller \cite{K}, John \cite{J} and Glassey \cite{G}.

Ignoring the Laplacian in \eqref{0-0} and solving the ordinary differential equation $\Phi_{tt}=|\Phi|^{p-1}\Phi$, one gets the exact, homogeneous in space, solution
\begin{equation}\label{0-1}
\Phi_T(t)=\frac{b_{0}}{(T-t)^{\alpha}}, \quad b_{0}:=\left[\frac{2(p+1)}{(p-1)^{2}}\right]^{\frac{1}{p-1}}, \quad T\in \RR,
\end{equation}
which blows up as $t\rightarrow T$.

This solution is stable 
  (see \cite{DonningerSchorkhuber16,DonningerSchorkhuber17,Donninger17}), and there are numerical evidences  that it also determines the leading order asymptotic of blow-up for generic blow-up solutions (see e.g. \cite{BCT}). This is indeed the case in dimension $N=1$, as proved by F.~Merle and H.~Zaag in a series of work (see \cite{MerleZaag08}, \cite{MerleZaag12} and references therein).

The solution $\Phi_T$ is an example of \emph{self-similar solution}, that is a solution of the equation which is invariant (up to a time translation) by the rescaling \eqref{0-2}. It is easy to see that a self-similar solution must be of the following form
\begin{equation}\label{0-3}
\Phi(t, x)=(T-t)^{-\alpha} u\left(\frac{x}{T-t}\right),
\end{equation}
for some function $u(y)$ and $T\in \RR$. Note that each self-similar solution with a regular profile $u$ provides an explicit example of regular initial data developing a singularity in finite time. Substituting the ansatz \eqref{0-3} into equation \eqref{0-0}, and assuming to simplify that the solution is radial, we obtain the following ordinary differential equation for the similarity profile $u(\rho)$:
\begin{equation}\label{PDE}
\left(1-\rho^{2}\right)u''+\left(\frac{N-1}{\rho}-2(\alpha+1)\rho\right)u'-\alpha(\alpha+1)u+|u|^{p-1}u=0.
\end{equation}
To our knowledge, the first theoretical work on the equation \eqref{PDE} is the article \cite{KW} by Kavian and Weissler, where a careful study of global solutions of \eqref{PDE} is carried out. It is proved in particular that \eqref{0-0} has no radial, finite energy, self-similar solution. However, this article does not contain any existence result.

Equation \eqref{PDE} has at least two explicit solutions: the regular, constant solution $b_0$ (corresponding to the blow-up solutions $\Phi_T$ of \eqref{0-0}), and the solution
\begin{equation}\label{singular}
u_{\infty}(\rho):=b_{\infty} \rho^{-\alpha}, \qquad b_{\infty}:=\left(\alpha(N-2-\alpha)\right)^{\frac{\alpha}{2}},
\end{equation}
singular at $\rho=0$,
which corresponds to the singular static solution $\Phi_{\infty}=b_{\infty}r^{-\alpha}$ of equation \eqref{0-0}.
In space dimension $N=3$, the existence of other regular solutions of \eqref{PDE} is proved in \cite{BBM} for $p=3$ and in \cite{BMW} in the energy supercritical case $p>5$ (for integer $p$). \ednote{{\color{red} Wei: In \cite{BMW}, the authors have only claimed their results for $N=3$ and $p\geq7$ odd integers. Should we mention that their results are valid for any integers $p>5$ or exactly as what the authors have claimed in \cite{BMW}? See also the abstract.}{\color{blue} Thomas: I have changed it to odd integer (also in the abstract). I think most of their proof works for any number $p>5$, and there is a remark somewhere in the article that the proof should work for any integer $p$. But they use in some places that $p\geq 7$ I think.}} In this last work, it is shown that for any integer $n$, there exists a regular solution of \eqref{PDE} that intersects $u_{\infty}$ exactly $n$ times on $(0,1)$. These solutions can be extended (except maybe for a finite number of them) to a regular solution of \eqref{PDE} on $(0,\infty)$ that goes to $0$ as $\rho\to\infty$. For the related issue of existence of self-similar solutions for energy-supercritical equivariant wave maps, see \cite{Shatah,BizonCMP}.\ednote{{\color{blue} Thomas: I added references to article proving the existence of wave maps.}}

In space dimension $N\geq 5$, an explicit solution of \eqref{PDE} was found by Glogi\'c and Sch\"orkhuber \cite{GS,GMS} in the cubic case $p=3$, which is also energy-supercritical. Namely:
\begin{equation}
 \label{SolGlSc}
 u(\rho)=\frac{2\sqrt{2(N-1)(N-4)}}{N-4+3\rho^2}.
\end{equation}
To our knowledge, and quite surprisingly, the works cited above are the only theoretical works on the existence of self-similar solutions for equation \eqref{0-0}. However heuristic arguments and numerical investigations (see \cite{R}) suggest the existence of a countable family of regular solutions of \eqref{PDE} for any supercritical nonlinearity $p>1+\frac{4}{N-2}$, below the Joseph-Lundgren exponent (defined below by \eqref{defpJL}) if $N\geq 11$.

In this work, we prove the existence of a countable family of regular solutions of \eqref{PDE} in higher dimension in the case $1+\frac{4}{N-2}< p\leq 1+\frac{4}{N-3}$. Our results are more precise when $p<1+\frac{4}{N-3}$:
\begin{thm}\label{Thm0}
Assume $N=3$, or $N\geq 4$ and $1+\frac{4}{N-2}<p<1+\frac{4}{N-3}$. For any non-negative integer $n\geq0$, there exists $\rho_n>1$ and an analytic positive function $u_n$ on $[0,\rho_n)$, which is solution of the equation \eqref{PDE} on $(0,\rho_n)$, and such that
$u_n-u_{\infty}$ has exactly $n+1$ zeros on $(0,1)$. Moreover,
$$ \lim_{n\to\infty} u_n(0)=+\infty.$$
\end{thm}

The conditions $p<$, $=$, $>1+\frac{4}{N-3}$ are equivalent to $b_{\infty}<$, $=$, $>b_{0}$ respectively, where $b_{\infty}$ is given by \eqref{singular}.
If $p<1+\frac{4}{N-3}$, as in Theorem \ref{Thm0}, the solutions $u_{\infty}$ and $b_0$ of \eqref{PDE} intersects on $(0,1)$. In the case $p=1+\frac{4}{N-3}$, we have $b_0=b_{\infty}$ and these two solutions intersect exactly at $\rho=1$. It is easy to see that in this case, any solution of \eqref{PDE} that is regular at $\rho=1$ must satisfy $u(1)=\{-b_0,b_0,0\}$.
In this case, our existence result is as follows:
\begin{thm}
\label{T:critical_case}
Assume $p=1+\frac{4}{N-3}$. Let $n_0$ be an integer and $C_0>0$. Then there exists $\rho_0>1$ and an analytic positive function $u$ on $[0,\rho_0)$, solution of \eqref{PDE} on $(0,\rho_0)$ and  such that $u(0)>C_0$, $u(1)=b_0$ and $u-u_{\infty}$ has at least $n_0$ zeros on $(0,1)$.
\end{thm}
The solutions constructed in Theorems \ref{Thm0} and \ref{T:critical_case} are defined in a neighborhood of $[0,1]$ in $[0,\infty)$, which corresponds to a solution $\Phi$ defined by \eqref{0-3} in a neighborhood of the wave cone $\{|x|<T-t\}$. Our next result is that the solutions of Theorems \ref{Thm0} and \ref{T:critical_case} are indeed global, except maybe for a finite number of them, and have a prescribed asymptotic behaviour as $\rho\to\infty$:
\begin{thm}\label{Thm1}
Assume $N\geq 3$ and $1+\frac{4}{N-2}<p\leq 1+\frac{4}{N-3}$.
There exists $C=C(N,p)>0$ with the following property. Let $u\in C^2([0,\sigma))$ ($\sigma>1$) be a solution of \eqref{PDE} on $(0,\sigma)$ and such that $u(0)\geq C$. Then $u$ is positive and can be extended to a positive analytic solution of \eqref{PDE} on $(0,+\infty)$.
Furthermore there exists $L>0$ such that
\begin{equation}
\label{asymptotics}
\lim_{\rho\to+\infty} \rho^{\alpha}u(\rho)=L\text{ and }
\lim_{\rho\to+\infty}\rho^{\alpha+1}u'(\rho)=-\alpha L.
\end{equation}
In particular, the solutions $u_n$ of Theorem \ref{Thm0} with $n$ large, and the solutions $u$ of Theorem \ref{T:critical_case} with $C_0$ large and $n_0\geq 1$ are global and satisfy \eqref{asymptotics} for some $L>0$.
\end{thm}
\begin{rem}
\label{R:N=3}
It is proved in \cite{BMW} that if $N=3$, \emph{all} solutions $u_n$ with $n$ odd satisfy the conclusion of Theorem \ref{Thm1} (see also Remark \ref{R:BMW} below).
\end{rem}
\begin{rem}
 Theorems \ref{Thm0}, \ref{T:critical_case} and \ref{Thm1} yield a countable family of global regular solutions of \eqref{PDE}, with the same asymptotics as $\rho\to\infty$  as the singular solution $u_{\infty}$ (up to a multiplicative constant). The case $p=3$, $N=5$ is covered by Theorem \ref{T:critical_case}.\ednote{{\color{red} Wei: The case $p=5$, $N=4$ is also covered by Theorem \ref{T:critical_case}. Should we mention it or not?}{\color{blue}Thomas: I mentionned the case $p=3$ her because we have the explicit solution of Glo-Scho, which gives a counter-example to the conclusion of Theorem 1.3 since the decay at infinity is different. Unfortunately $p=3$ never satisfies $1+\frac{4}{N-2}<p<1+\frac{4}{N-3}$}.} In this case, the solutions satisfying the assumptions of Theorem \ref{Thm1} are of order $1/r$ at infinity. The explicit solution \eqref{SolGlSc} of Glogi\'c and Sch\"orkhuber is of order $1/r^2$ at infinity. This proves that the assumption $u(0)\geq C$ in Theorem \ref{Thm1} cannot be removed in general.
 \end{rem}
 \begin{rem}
 \label{R:large_n}
 Assume $N=3$ or $N\geq 4$ and  $1+\frac{4}{N-2}<p<1+\frac{4}{N-3}$. As a consequence of the proof of Theorem \ref{Thm0}, the self-similar solutions $u_n$ constructed in these theorems get close for large $n$, on the interval $(0,1)$, to the singular self-similar solution $u_{\infty}$. More precisely,
  \begin{gather*}
  \forall 0<\rho_1<1 \quad
  \sup_{\rho_1<\rho\leq 1}|u_n(\rho)-u_{\infty}(\rho)|\underset{n\to\infty}{\longrightarrow}0,\\
  \forall 0<\rho_1<\rho_2<1,\quad \sup_{\rho_1<\rho<\rho_2}|u_n'(\rho)-u_{\infty}'(\rho)|+
   |u_n'(1)-u_{\infty}'(1)|\underset{n\to\infty}{\longrightarrow}0.
  \end{gather*}

 Similar statements hold in the case $N\geq 4$, $p=1+\frac{4}{N-3}$, letting $C_0\to\infty$ in Theorem \ref{T:critical_case}. See Corollary \ref{Cor0} and Proposition \ref{P:limitU'}.
 \end{rem}

\begin{rem}
When the profile $u$ is global, as in Theorem \ref{Thm1}, the corresponding solution $\Phi$ defined by \eqref{0-3} is also global in one time direction. Indeed, if $u$ satisfies the assumptions of Theorem \ref{Thm1}, then $\Phi$, defined by \eqref{0-3} yields a self-similar solution of \eqref{0-0} on $(-\infty,T)\times \RR^N$ and on $(T,\infty)\times \RR^N$, which is, at fixed $t$, of order $1/r^{\alpha}$ for $r$ large. Note that $(\Phi(t),\partial_t\Phi(t))$ belongs to all the homogeneous Sobolev spaces $\dot{H}^{s}\times \dot{H}^{s-1}$, $s>s_c$, but misses the critical Sobolev space $\dot{H}^{s_c}\times \dot{H}^{s_c-1}$ by a logarithm.
Let us mention that the fact that the initial data of exact radial self-similar solutions of \eqref{0-0} are not in $\dot{H}^{s_c}\times \dot{H}^{s_c-1}$ is general and is indeed a consequence of \cite[Theorem 3.1]{KW} (recalled below in Theorem \ref{T:A50}). 

Let us also mention the work \cite{KS}, where a global, nonscattering solution of \eqref{0-0} is constructed in the case $N=3$, $p=7$ for focusing \emph{and defocusing} nonlinearity, by regularization of a singular self-similar solution. The initial data of this solution has the same asymptotic behaviour in $1/r^{\alpha}$ for large $r$, and also misses the critical Sobolev space by a logarithmic factor. 
\end{rem}

\subsection*{Sketch of proof and outline of the article}
The proofs of Theorem \ref{Thm0} and \ref{T:critical_case} are based on a refinement of the classical shooting method, and use some of the ideas of the proof of the corresponding result in space dimension $3$ sketched in the appendix of \cite{BMW} (see also \cite{L}).\ednote{Thomas: I added a reference to the article of Lepin. Unfortunately I could not find the english version of this article online, only the russian one. Did you find it? {\color{red} Wei: Until now, I can only find and download the russian version. I will try again to find the English version. Perhaps Bizo\'{n} can read the russian version.}{\color{blue} Thomas: certainly he has grown up in Poland before the collapse of Soviet Union!}}


For any $c>0$, one can define a regular solution $u(\rho,c)$ of \eqref{PDE} such that $u(0,c)=c$, $u'(0,c)=0$. It is easy to check that this solution can be extended to $[0,1)$, and that $u(\rho,c)$ has a limit $u(1,c)$ as $\rho\to 1$. This is not sufficient however to extended $u(\cdot,c)$ to a regular solution of \eqref{PDE} in a neighborhood of $\rho=1$, as the derivative $u'(\rho,c)$ might diverge logarithmically as $\rho \to 1$.

Similarly, if $p<1+\frac{4}{N-3}$, for any $b>0$, one can define a regular solution $U(\rho,b)$ of \eqref{PDE} such that $U(1,b)=b$, and that might be extended to $(0,1]$, but not, in general to a regular solution at $\rho=0$.
To obtain the conclusion of Theorem \ref{Thm0}, we must prove that there exists $c_n$ such that $u(\cdot,c_n)-u_{\infty}$ has exactly $n+1$ zeros on $(0,1)$, and $u(\cdot,c_n)$ coincide with a solution $U(\cdot,b_n)$ for some $b_n$. The two main ingredients of this proof are the following two facts, proved in Subsections \ref{SS:comparison}  and \ref{SS:intersection} respectively:
\begin{gather}
 \label{I:binfty} \displaystyle \lim_{c\to\infty} u(1,c)=b_{\infty}\\
 \label{I:zeros} \text{as }c\to +\infty\text{ the number of zeros of  }u(\cdot,c)-u_{\infty}\text{ goes to infinity}.  \end{gather}
 (see Corollary \ref{Cor0} for \eqref{I:binfty} and Proposition \ref{prop0} for \eqref{I:zeros}).
The rigorous proof of \eqref{I:binfty} in the appendix of \cite{BMW} is specific to dimension $N=3$, and we give a different proof, close to the ``physicist'' proof sketched in Section 3 of \cite{BMW}.
The proof of \eqref{I:zeros} is a non-trivial adaptation of \cite[Lemma 2]{BMW} where it is proved in dimension $3$ for $N\geq 7$.
We  note that points \eqref{I:binfty} and \eqref{I:zeros} hold for any $p>1+\frac{4}{N-2}$, with the additional condition $p<p_{JL}$ if $N\geq 11$ for \eqref{I:zeros} ($p_{JL}$ is the Joseph-Lundgren exponent, see \cite{JL} and \eqref{defpJL} below). The assumption $p\leq 1+\frac{4}{N-3}$ is not needed here.

Once points  \eqref{I:binfty} and \eqref{I:zeros} are known, Theorem \ref{Thm0} can be proved by a geometrical argument, similar to the one of the proof of the particular case $N=3$ in \cite{BMW}. We give the details,  adopting a slightly different point of view than in \cite{BMW} and including the proof of the case where $n$ is even, which is omitted there.

The end of the proof of Theorem \ref{T:critical_case} in Section \ref{S:critical_case}, also relying on points \eqref{I:binfty} and \eqref{I:zeros}, is quite different. Indeed, the local well-posedness at $\rho=1$ for the case $\alpha=1+\frac{4}{N-3}$ takes a completely different form (see Proposition \ref{P:local2}). In this case, we prove that for any $a\in \RR$, there exists a unique solution $U(\cdot,a)$ of \eqref{PDE} in a neighborhood of $\rho=1$ such that $U(1,a)=b_0=b_{\infty}$ and $U'(1,a)=a$. As a consequence of the uniqueness in this well-posedness statement, we also obtain that a solution $u(\rho)$ of \eqref{PDE} defined on $(0,1)$ can be extended to a regular solution if and only if $\lim_{\rho \to 1} u(\rho)=b_{\infty}$ (see Lemma \ref{L:C40}). Once this is known, the conclusion of Theorem \ref{T:critical_case} follows quite easily from \eqref{I:binfty}, \eqref{I:zeros} and an argument related to the intermediate value theorem.

Section \ref{S:extension} is dedicated to the proof of Theorem \ref{Thm1}. In the case $N=3$, it is contained in \cite{BMW} for the global existence of the solution and \cite{KW} for its asymptotic behaviour. The proof relies crucially on a simple monotonicity formula, that seems to be specific to $N=3$.
The proof of Theorem \ref{Thm1} in higher dimension is more intricate, and is based on a refinement of the method of \cite{KW}: see Propositions \ref{P:A0} and \ref{P:A1} which complete some of the results of \cite{KW}. Let us mention that the proof is still more complicated in the case $p=1+\frac{4}{N-3}$, where we need the extension of \eqref{I:binfty} mentioned in Remark \ref{R:large_n}.
\subsection*{Acknowledgment}
The second author would like to thank Piotr Bizo\'n, Birgit Sch\"orkhuber and Fred Weissler for fruitful discussions.
\section{Properties of solutions that are regular at the origin}
\label{S:regular}
In this section, the first part of the proof of the Theorems \ref{Thm0} and \ref{T:critical_case}, we study solutions of \eqref{PDE} that are regular at $\rho=0$. Subsection \ref{SS:0} concerns the local existence. In Subsections \ref{SS:comparison} and \ref{SS:intersection} we compare the solutions $u$ of \eqref{PDE} such that $u(0)$ is large to the singular solution $u_{\infty}$, proving the two crucial properties \eqref{I:binfty} and \eqref{I:zeros} mentioned in the introduction.

Throughout this section, we assume $N\geq 3$, $1+\frac{4}{N-2}<p$.

\subsection{Existence of solutions on $[0,1)$}
\label{SS:0}
In the first step, we prove the local existence of a one-parameter family of regular solutions in a neighborhood of $\rho=0$ and prove that these solutions can be extended to $[0,1)$. As in all the article, we denote by $'$ the derivative with respect to the variable $\rho$.
\begin{prop}
 \label{P:local0}
 Assume $1+\frac{4}{N-2}<p$. Let $c\in \RR$. There exists a unique function $u(\cdot,c)\in C^2([0,1),\RR)$ which is solution of \eqref{PDE} on $(0,1)$, and such that $u(0,c)=c$, $u'(0,c)=0$. The functions $u$ and $u'$ are continuous on $[0,1)\times \RR$  and for all $c>0$, $\rho\in [0,1)$, $u(\rho,c)$ is positive. Furthermore, $u$ can be extended to a continuous function on $[0,1]\times \RR$.
\end{prop}

\begin{rem}
 It follows from the analyticity of $u\mapsto |u|^{p-1}u$ away from $u=0$ and from the positivity of $u$ that the solution $\rho\mapsto u(\rho,c)$ is indeed analytic on $[0,1)$.
\end{rem}

\begin{proof}[Proof of Proposition \ref{P:local0}]

$ $

\smallskip

\noindent\emph{Step 1. Local Cauchy theory at the origin.}
Writing the equation \eqref{PDE} in self-adjoint form:
\begin{equation}\label{2-2-5}
\left(\rho^{N-1}\left(1-\rho^{2}\right)^{\alpha-\frac{N-3}{2}}u'\right)'=\rho^{N-1}\left(1-\rho^{2}\right)^{\alpha-\frac{N-1}{2}}
u\left(b_{0}^{p-1}-|u|^{p-1}\right),
\end{equation}
the proof of the local existence with initial data at $\rho=0$ can be treated by solving by fixed point the equation
$$u(\rho)=c+\int_0^{\rho}\int_0^{\tau} \left( \frac{\sigma}{\tau} \right)^{N-1}\left( \frac{1-\sigma^2}{1-\tau^2} \right)^{\alpha-\frac{N-1}{2}} u(\sigma)\left( b_0^{p-1}-|u(\sigma)|^{p-1} \right)d\sigma (1-\tau^2)^{-1}d\tau.$$
One obtains that for all $c \in \RR$, there exists $\rho_c\in (0,1)$, and a unique function $u(\cdot,c)\in C^2([0,\rho_c),\RR)$, that satisfies equation \eqref{PDE} on $(0,\rho_c)$ and such that $u(0,c)=c$ and $u'(0,c)=0$. The function $(\rho,c)\mapsto u(\rho,c)$ is continuous in a neighborhood of $\{\rho=0\}$.

We omit the classical details and refer to the two first steps of the proof of Proposition \ref{P:local1} for the proof of a similar result (see also \cite[Section 2]{BMW}).

\medskip

\noindent\emph{Step 2. Extension of the solution.}
In this step, we prove that the solution $u(\cdot,c)$ can be extended to $[0,1)$. Noting that away from the boundary points $\rho=0$ and $\rho=1$, \eqref{PDE} is a regular ordinary differential equation, we see that it is sufficient to prove that $u$ remains bounded on the intersection of its maximal interval of existence with $[0,1)$.
For this we use the Lyapunov functional (see \cite{KW,H}) defined by
\begin{equation}\label{2-3-0}
H(\rho)=\left(1-\rho^{2}\right) \frac{(u^{\prime})^{2}}{2}+\frac{|u|^{p+1}}{p+1}-\frac{p+1}{(p-1)^{2}}u^{2},
\end{equation}
which satisfies, for $\rho\in (0,1)$ in the domain of existence of $u$,
\begin{equation}\label{2-3-1}
H(\rho)\geq-\frac{1}{p-1}\left(\frac{2(p+1)}{(p-1)^{2}}\right)^{\frac{2}{p-1}} \quad \text { and } \quad H^{\prime}(\rho)=\left(\frac{p+3}{p-1} \rho-\frac{N-1}{\rho}\right)(u^{\prime})^{2}\leq 0,
\end{equation}
where we have used, to obtain the second inequality, the assumption $p>1+\frac{4}{N-2}$.
Therefore if $u=u(\rho,c)$, we get
\begin{equation}\label{2-3-2}
H(\rho) \leq H(0)=\frac{c^{p+1}}{p+1}-\frac{(p+1) c^{2}}{(p-1)^{2}},
\end{equation}
which implies
\begin{equation}\label{2-3-3}
\sqrt{1-\rho^{2}}\left|u^{\prime}(\rho)\right| \leq c^{\frac{p+1}{2}}
\end{equation}
and $u(\rho)$ is bounded for $0<\rho<1$. Indeed, the function $F:[0,+\infty)\to \RR$ defined by
$$F(u)=\frac{u^{p+1}}{p+1}-\frac{p+1}{(p-1)^2}u^2,$$
which satisfies $F(0)=0$, is decreasing on $[0,b_0]$ and increasing on $[b_0,+\infty)$. We thus obtain that if $c \geq b_{0}$,
\begin{equation}\label{2-3-4}
|u(\rho)| \leq c \qquad \forall \rho\in (0,1),
\end{equation}
and if $0<c<b_0$, then
\begin{equation}\label{2-3-4bis}
|u(\rho)| \leq \overline{c} \qquad \forall \rho\in (0,1),
\end{equation}
where $\overline{c}$ is the unique number in $[b_0,+\infty)$ such that $F(c)=F(\overline{c})$.

\medskip

\emph{Step 3. Continuity.}
We next prove that $u(\rho,c)$ has a limit as $\rho\to 1$, and that $u$, extended to $\rho=1$, is continuous on $[0,1]\times \RR$.
By \eqref{2-3-4}, \eqref{2-3-4bis}, and the self-adjoint form \eqref{2-2-5} of equation \eqref{PDE}, we obtain that for all $M>0$, there exists a constant $K_M$ such that
\begin{equation}
\label{bound_deriv}
\forall \rho\in [0,1),\; \forall c\in [-M,M],\quad \left|\left(\rho^{N-1}\left(1-\rho^{2}\right)^{\alpha-\frac{N-3}{2}}u'\right)'\right|\leq K_M(1-\rho)^{\alpha-\frac{N-1}{2}}.
\end{equation}
In the case $p<1+\frac{4}{N-3}$, we have $\frac{N-3}{2}<\alpha<\frac{N-1}{2}$, and thus the right-hand side of \eqref{bound_deriv} is integrable. We deduce that for an absolute constant $C$,
\begin{equation}
\label{SAbound}
\forall \rho\in [1/2,1),\; \forall c\in [-M,M],\quad \left|u'(\rho,c)\right|\leq CK_M \left(1-\rho\right)^{\frac{N-3}{2}-\alpha},
\end{equation}
and the fact that $u(\rho,c)$ has a limit $u(1,c)$ as $\rho\to 1$ follows easily. In the case $p=1+\frac{4}{N-3}$, $\alpha=\frac{N-3}{2}$ and \eqref{SAbound} must be replaced by
\begin{equation}
\label{SAbound'}
\forall \rho\in [1/2,1),\; \forall c\in [-M,M],\quad \left|u'(\rho,c)\right|\leq CK_M |\log(1-\rho)|,
\end{equation}
which implies again that $u(\rho,c)$ has a limit as $\rho\to 1$. Finally, in the case $p>1+\frac{4}{N-3}$, we deduce from \eqref{bound_deriv}
$$\forall \rho\in [0,1),\; \rho^{N-1}(1-\rho)^{\alpha-\frac{N-3}{2}} |u'(\rho,c)|\leq K_M\int_0^{\rho} (1-\sigma)^{\alpha-\frac{N-1}{2}}d\sigma\leq C K_M(1-\rho)^{\alpha-\frac{N-3}{2}},$$
and we obtain
\begin{equation}
\label{SAbound''}
\forall \rho\in [1/2,1),\; \forall c\in [-M,M],\quad \left|u'(\rho,c)\right|\leq CK_M.
\end{equation}
Again, the fact that $u(\rho,c)$ has a limit as $\rho\to 1$ follows easily.

Since \eqref{PDE} is a regular ordinary differential equation on $(0,1)$ and that $u$ is continuous at $\rho=0$ by construction, we deduce by standard theory that $(\rho,c)\mapsto u(\rho,c)$ is continuous on $[0,1)\times \RR$.  We next prove the continuity of this function on $[0,1]\times \RR$. It is sufficient to prove the continuity at any point $(1,c)$, $c\in \RR$. Let $(\rho_k)_k\in [0,1]^{\NN}$, $(c_k)\in \RR^{\NN}$ such that
$$\lim_k \rho_k=1\quad \lim_k c_k=c.$$
Let $M=\sup_{k} |c_k|$. Let $\eps>0$. Fix $\rho_{\eps}\in (1/2,1)$ such that
\begin{equation}
\label{cases}
 \begin{cases}
  \int_{\rho_{\eps}}^1 \left(1-\rho\right)^{\frac{N-3}{2}-\alpha}d\rho\leq \frac{\eps}{CK_M} &\text{ if }p<1+\frac{4}{N-3}\\
  \int_{\rho_{\eps}}^1 \left|\log \left(1-\rho\right)\right| \leq \frac{\eps}{CK_M}&\text{ if }p=1+\frac{4}{N-3}\\
  1-\rho_{\eps}\leq \frac{\eps}{CK_M}&\text{ if }p>1+\frac{4}{N-4}.
 \end{cases}
\end{equation}
where $CK_M$ is as in \eqref{SAbound}, \eqref{SAbound'} or \eqref{SAbound''}. Then
$$|u(\rho_k,c_k)-u(1,c)|\leq |u(\rho_k,c_k)-u(\rho_{\eps},c_k)|+|u(\rho_\eps,c_k)-u(\rho_{\eps},c)|+|u(\rho_{\eps},c)-u(1,c)|.$$
By the continuity of $u$ on $[0,1)\times \RR$,
$$\lim_{k\to+\infty} |u(\rho_{\eps},c_k)-u(\rho_{\eps},c)|=0.$$
By \eqref{SAbound}, \eqref{SAbound'} or \eqref{SAbound''} and \eqref{cases}, we obtain that for large $k$,
$$|u(\rho_k,c_k)-u(\rho_{\eps},c_k)|+|u(\rho_{\eps},c)-u(1,c)|\leq 2\eps.$$
Thus
$$\limsup_{k\to\infty}\left|u(\rho_{k},c_k)-u(1,c)\right|\leq 2\eps,$$
and since $\eps$ is arbitrary,
$$\lim_{k\to\infty}u(\rho_k,c_k)=u(1,c),$$
concluding the proof of the continuity of $(\rho,c)\mapsto u(\rho,c)$.

\medskip \emph{Step 4. Positivity.}
To conclude the proof of Proposition \ref{prop1}, we must prove that $u(\cdot,c)$ is positive for $c>0$. For any $0<\rho\leq1$ and $c\geq0$, we define
\begin{equation}
\label{def}
v(\rho,c):=\frac{u(\rho,c)}{u_{\infty}(\rho)},
\end{equation}
and the
Lyapunov function
\begin{equation}\label{2-2-3}
H_{v}(\rho,c):=\frac{1}{2}\rho^{2}\left(1-\rho^{2}\right)(v')^{2}-\alpha(N-2-\alpha)\left(\frac{v^{2}}{2}-\frac{|v|^{p+1}}{p+1}\right).
\end{equation}
Since by our assumptions $\alpha:=\frac{2}{p-1}<\frac{N-2}{2}$, we obtain from equation \eqref{PDE} and the definition of $v(\rho,c)$,
\begin{equation}\label{2-2-4}
H_{v}'(\rho,c)=(2\alpha-N+2)\rho(v')^{2}\leq0,
\end{equation}
hence $H_v$ is decreasing with respect to $\rho\in(0,1]$. Furthermore, by the definition of $v$, $H_v(0,c)=0$ and $v'(\rho,c)\neq 0$ for $c>0$ and small $\rho>0$. As a consequence, $H_v(\rho,c)<0$ for all $\rho\in (0,1]$, which implies that $0<v(\rho,c)<\left(\frac{p+1}{2}\right)^{\frac{1}{p-1}}$ and thus
\begin{equation}\label{2-2-29}
  0<u(\rho,c)<\left(\frac{p+1}{2}\right)^{\frac{1}{p-1}}b_{\infty}\rho^{-\alpha}, \qquad \forall \,\, c>0, \,\, \,\, \forall \,\, \rho\in(0,1].
\end{equation}

\end{proof}

\subsection{Convergence to the singular solution}
\label{SS:comparison}


Recall from \eqref{def} and \eqref{2-2-3} in the proof of Proposition \ref{P:local0} the definition of $v$ and $H_v$.

In this section, we prove the following asymptotic upper bound for $H_{v}$.
\begin{prop}
 \label{P:Thomas}
 Assume $p>1+\frac{4}{N-2}$.
 For all $0<\eps<1$, there exists $M:=M(\eps)$ such that
 \begin{equation}
  \label{T00}
  \lim_{c\to+\infty} H_v\left( \frac{M}{c^{\frac{p-1}{2}}},c \right)\leq -\alpha(N-2-\alpha)\left( \frac{(1-\eps)^2}{2} -\frac{(1-\eps)^{p+1}}{p+1}\right).
 \end{equation}
\end{prop}
As a consequence, we obtain:
\begin{cor}\label{Cor0}
Let $\sigma_0,\sigma_1$ such that $0<\sigma_0<\sigma_1<1$. Then
\begin{gather}\label{5-4'}
  \lim_{c\rightarrow+\infty} \sup_{\sigma_0\leq \rho\leq 1}|v(\rho,c)-1|=0\\
  \label{5-5}
  \lim_{c\rightarrow+\infty} \sup_{\sigma_0<\rho\leq \sigma_1}|v'(\rho,c)|=0.
\end{gather}
In particular:
\begin{equation}
 \label{5-4}
 \lim_{c\rightarrow +\infty} u(1,c)=b_{\infty}.
\end{equation}
\end{cor}
Let us mention that only \eqref{5-4} is needed for the proof of Theorems \ref{Thm0} and \ref{T:critical_case}. The properties \eqref{5-4'} and \eqref{5-5} will be used  for the proof of Proposition \ref{P:limitU'}, which is crucial to obtain Theorem \ref{Thm1} in the case $p=1+\frac{4}{N-3}$.
\begin{proof}[Proof of Proposition \ref{P:Thomas}]
We use the self-adjoint form \eqref{2-2-5} of the equation \eqref{PDE}.
Let $Q$ be the regular solution of
\begin{equation}\label{T30}
\left\{
\begin{aligned}
 \left( \frac{d^2}{dr^2}+\frac{N-1}{r}\frac{d}{dr} \right)Q+|Q|^{p-1}Q&=0,\\
 Q(0)=1,\quad Q'(0)&=0.
\end{aligned}
\right.
\end{equation}
The existence of $Q$ is classical (see e.g. Proposition 1 of \cite{BFM} or \cite{JL}). Let $V(r):=\frac{r^{\alpha}}{b_{\infty}}Q$ and
\begin{equation}\label{T31}
  E(r)=\frac{1}{2}r^2(V')^{2}-\alpha(N-2-\alpha)\left( \frac{V^2}{2}-\frac{|V|^{p+1}}{p+1} \right),
\end{equation}
which is a decreasing function of $r$ (see \eqref{E'} below).
We claim
\begin{gather}
 \label{T01}
 \lim_{c\to+\infty} H_v\left( \frac{r}{c^{\frac{p-1}{2}}},c\right)=E(r),\quad \forall \, r>0;\\
 \label{T02}
 \lim_{r\to+\infty} E(r)=-\alpha(N-2-\alpha)\left( \frac{1}{2}-\frac{1}{p+1} \right).
\end{gather}
Noting that $V\mapsto -\alpha(N-2-\alpha)\left( \frac{V^2}{2}-\frac{|V|^{p+1}}{p+1} \right)$ is a decreasing function of $V\in [0,1]$, we see that \eqref{T01} and \eqref{T02} readily imply the conclusion \eqref{T00} of the proposition.

\medskip

\noindent\emph{Proof of \eqref{T01}}.

We start by a change of variable. Let $r=c^{\frac{p-1}{2}}\rho$. Define $\tilde{u}$ by $u(\rho)=c \tilde{u} \left( c^{\frac{p-1}{2}} \rho\right)$. Then $\tilde{u}$ satisfies the following equation
\begin{equation}\label{T10}
 \left( 1-\frac{r^2}{c^{p-1}} \right)\tu''+\left( \frac{N-1}{r}-2(\alpha+1)\frac{r}{c^{p-1}} \right)\tilde{u}'-\frac{\alpha(\alpha+1)}{c^{p-1}}\tilde{u}+|\tilde{u}|^{p-1}\tilde{u}=0.
\end{equation}
We fix $M>0$ arbitrarily. By the bounds
\eqref{2-3-3} and \eqref{2-3-4}, we have, for large $c$,
\begin{equation}\label{T11}
 |\tilde{u}(r)|\leq 1,\quad |\tilde{u}'(r)|\leq \frac{3}{2},\quad \forall \, r\in [0,M].
\end{equation}
Combining with the equation \eqref{T10}, we deduce that, for large $c$,
\begin{equation}\label{T12}
 \left|\tu''(r)\right|\leq \frac{2(N-1)}{r}+2,\quad \forall \, r\in [0,M].
\end{equation}
Denoting by $K$ a large constant, independent of $M$ and that may change from line to line, we obtain that, for large $c$,
\begin{eqnarray}\label{T13}
&& \left| \frac{d}{dr}\left( r^{N-1}\frac{d}{dr}\left( \tu-Q \right)\right) +r^{N-1}\left( |\tu|^{p-1}\tu-|Q|^{p-1}Q \right)\right| \\
\nonumber &\leq& \frac{K}{c^{p-1}}r^{N-1}(1+M^2),\qquad \forall \, r\in [0,M].
\end{eqnarray}
Let $D(r)=\max_{0\leq s\leq r} \left|\tu(s)-Q(s)\right|$. Integrating \eqref{T13} and using the bound \eqref{T11}, we obtain
\begin{equation}\label{T14}
 \left|r^{N-1} \frac{d}{dr} (\tu-Q)\right|\leq K\int_0^{r} \theta^{N-1}D(\theta)\,d\theta+\frac{K}{c^{p-1}}r^N(1+M^2),\quad \forall \, r\in [0,M].
\end{equation}
Hence, for large $c$,
\begin{align}\label{T32}
 \left|\tu(r)-Q(r)\right|&\leq K\int_0^{r} \int_0^{s} \frac{\theta^{N-1}}{s^{N-1}}D(\theta)\,d\theta\,ds+\frac{K}{c^{p-1}}r^{2}(1+M^2)\\
 \nonumber &\leq K\int_0^r sD(s)\,ds +\frac{Kr^{2}}{c^{p-1}}(1+M^2), \quad \forall \, r\in [0,M],
\end{align}
where at the last line we have used that $\frac{\theta}{s}\leq 1$ and $D(\theta)\leq D(s)$ when $0\leq \theta\leq s$. The preceding bound yields:
\begin{equation}\label{T33}
  D(r)\leq K\int_0^{r} sD(s)\,ds+\frac{Kr^{2}}{c^{p-1}}(1+M^2),\quad \forall \, r\in [0,M].
\end{equation}
By Gr\"onwall's Lemma, we get, for large $c$,
\begin{equation}\label{T34}
  D(r)\leq \frac{K M^{2}(1+M^{2})}{c^{p-1}} e^{K M^2},\quad \forall \, r\in [0,M].
\end{equation}
In view of \eqref{T14} and \eqref{T34}, we deduce
\begin{equation}\label{T20}
 \lim_{c\to\infty}
 \sup_{0\leq r\leq M}\left(|\tilde{u}(r)-Q(r)|+|\tilde{u}'(r)-Q'(r)|\right)=0.
\end{equation}
Let $\tilde{v}=\frac{r^{\alpha}}{b_{\infty}} \tu=v\left( c^{-\frac{p-1}{2}}r \right)$. Then
\begin{equation}\label{T35}
  H_v\left(\frac{M}{c^{\frac{p-1}{2}}},c\right)=\frac{M^2}{2}\left( 1-\frac{M^2}{c^{p-1}} \right)(\tilde{v}')^2\left( M  \right)-\alpha(N-2-\alpha)\left( \frac{\tilde{v}^2\left( M\right)}{2}-\frac{|\tilde{v}|^{p+1}\left( M\right)}{p+1} \right).
\end{equation}
Since by \eqref{T20},
\begin{equation}\label{T36}
  \lim_{c\to\infty} \tilde{v}(M)=V(M)>0,\qquad \lim_{c\to\infty}\tilde{v}'(M)=V'(M),
\end{equation}
we deduce
\begin{equation}\label{T37}
  \lim_{c\to\infty} H_v\left(\frac{M}{c^{\frac{p-1}{2}}},c\right)=\frac{M^2}{2}(V')^2( M)-\alpha(N-2-\alpha)\left( \frac{V^2( M)}{2}-\frac{\left|V( M)\right|^{p+1}}{p+1} \right),
\end{equation}
which is exactly \eqref{T01} since $M>0$ is arbitrary.

\medskip

\noindent\emph{Proof of \eqref{T02}}.

We will prove
\begin{equation}\label{limVV'}
\lim_{r\to\infty}rV'(r)=0,\qquad \lim_{r\to\infty}V(r)=1.
\end{equation}
This is well-known (see e.g. \cite{JL}). We give a proof for the sake of completeness. We have
\begin{equation}\label{E'}
E'(r)=-r(N-2-2\alpha)\left(V'(r)\right)^2\leq0, \qquad E(0)=0.
\end{equation}
Since $E$ is bounded from below, it has a limit $E_{\infty}\leq 0$ at infinity. We note that $E_{\infty}$ cannot be $0$, or else $V'(r)$ would be $0$ for all $r>0$, and thus $V$ (and $Q$) would be $0$, a contradiction with the definition of $Q$.
Note also that by the bound $E(r)<0$ for $r>0$, $V$ cannot vanish which implies
$$\forall r>0,\quad V(r)>0.$$

Since $E$ is bounded, $V$ is also bounded and, by \eqref{E'} we see that
$\int_0^{+\infty} r(V'(r))^2\,dr$ is finite. Let $r_n\to+\infty$, such that there exists $V_{\infty}\in \Real$ with
\begin{equation}\label{T38}
  \lim_{n\to\infty} V(r_n)=V_{\infty}\geq 0.
\end{equation}
If $r_n\leq r\leq 2r_n$,
\begin{equation}\label{T39}
  |V(r_n)-V(r)|\leq \int_{r_n}^{2r_n}|V'(r)|\,dr\leq \left(\int_{r_n}^{2r_n} \frac{dr}{r}\right)^{\frac{1}{2}}\left(\int_{r_n}^{2r_n}\left|V'(r)\right|^2r\,dr\right)^{\frac{1}{2}}\underset{n\to\infty}{\longrightarrow} 0.
\end{equation}
As a consequence,
\begin{equation}\label{T41}
  \lim_{n\to\infty}\left(\sup_{r_n\leq r\leq 2r_n} |V(r)-V_{\infty}|\right)=0.
\end{equation}
By the definition of $E$, one has
\begin{equation}\label{T42}
  \int_{r_n}^{2r_n} \frac{1}{r}E(r)\,dr=\frac{1}{2}\int_{r_n}^{2r_n} r \left(V'(r)\right)^2dr+\alpha(N-2-\alpha)\int_{r_n}^{2r_n} \frac{1}{r}\left(\frac{|V|^{p+1}}{p+1}-\frac{V^2}{2} \right)dr.
\end{equation}
Letting $n\to\infty$, we deduce, since $\lim_n\int_{r_n}^{2r_n} r \left(V'(r)\right)^2dr=0$,
\begin{equation}\label{T40}
 E_{\infty}=\alpha(N-2-\alpha)\left( \frac{V_{\infty}^{p+1}}{p+1}-\frac{V_{\infty}^2}{2} \right).
\end{equation}
As a consequence, $V_{\infty}$ is independent of the choice of the sequence $r_n$, which proves
\begin{equation}\label{T43}
  \lim_{r\to\infty} V(r)=V_{\infty}.
\end{equation}
Since $E_{\infty}<0$, we know that $V_{\infty}$ is nonzero. By the definition of the energy and \eqref{T40},
\begin{equation}\label{T44}
  \lim_{r\to\infty} r^2(V'(r))^2=0.
\end{equation}
We next prove by contradiction that $V_{\infty}=1$. Assume that $V_{\infty}\neq 1$. Thus $V_{\infty}-|V_{\infty}|^{p-1}V_{\infty}\neq 0$.
Integrating the equation:
\begin{equation}\label{T45}
  \frac{d}{dr} \left( r^{N-1-2\alpha} \frac{dV}{dr} \right)=r^{N-1-2\alpha} \frac{\alpha(N-2-\alpha)}{r^2} \left( V-|V|^{p-1}V \right)
\end{equation}
between $1$ and $r\gg 1$, and recalling our assumption $\alpha<\frac{N-2}{2}$, we obtain
\begin{equation}\label{T46}
  r^{N-1-2\alpha}\frac{dV}{dr}(r)-\frac{dV}{dr}(1)\sim \alpha(N-2-\alpha)\left( V_{\infty}-|V_{\infty}|^{p-1}V_{\infty} \right)\frac{r^{N-2-2\alpha}}{N-2-2\alpha},\quad r\to\infty.
\end{equation}
Hence, we arrive at
\begin{equation}\label{T47}
  \frac{dV}{dr}(r)\sim \frac{\alpha(N-2-\alpha)\left(V_{\infty}-|V_{\infty}|^{p-1}V_{\infty} \right)}{(N-2-2\alpha)\,r},\quad r\to\infty,
\end{equation}
a contradiction with the integrability of $r(V')^2(r)$. This proves \eqref{limVV'}, and thus \eqref{T02}, concluding the proof of the proposition.
\end{proof}

\begin{proof}[Proof of Corollary \ref{Cor0}]
By Proposition \ref{P:Thomas}, for any $0<\eps<\frac{1}{2}$, there exists $M:=M(\eps)\gg 1$ and $c(\eps)\gg 1$ such that, for any $c>c(\eps)$,
\begin{equation}\label{5-0}
  H_v\left(\frac{M}{c^{\frac{p-1}{2}}},c\right)\leq -\alpha(N-2-\alpha)\left[\frac{(1-2\eps)^2}{2}-\frac{(1-2\eps)^{p+1}}{p+1}\right].
\end{equation}
Since $H_v$ is decreasing with respect to $\rho\in(0,1]$, by the definition \eqref{2-3-0} of $H_{v}$ and \eqref{5-0}, we have, for any  $c>c(\eps)$ and $\frac{M}{c^{\frac{p-1}{2}}}\leq\rho<1$,
\begin{equation}\label{5-1}
  \frac{1}{2}\rho^2(1-\rho^2){v'}^2-\alpha(N-2-\alpha)\left[\frac{v^2}{2}-\frac{v^{p+1}}{p+1}\right]\leq-\alpha(N-2-\alpha)\left[\frac{(1-2\eps)^2}{2}-\frac{(1-2\eps)^{p+1}}{p+1}\right].
\end{equation}
In particular,
\begin{equation}\label{5-2}
  \frac{1}{2}v^2-\frac{1}{p+1}|v|^{p+1}\geq\frac{(1-2\eps)^2}{2}-\frac{(1-2\eps)^{p+1}}{p+1}.
\end{equation}
Let $G(v)=\frac{1}{2}v^2-\frac{1}{p+1}v^{p+1}$, which is monotonically increasing on $(0,1)$ and decreasing on $(1,+\infty)$. For all small $\eps$, let $r_{\eps}$ be the unique number larger than $1$ such that $G(1-2\eps)=G(1+r_{\eps})$, and note that $r_{\eps}$ tends to $0$ as $\eps$ tends to $0$. By \eqref{5-2}, for $\frac{M}{c^{\frac{p-1}{2}}}\leq \rho< 1$,
\begin{equation}\label{5-3}
  (1-2\eps)\leq v(\rho)\leq 1+r_{\eps},
\end{equation}
which yields \eqref{5-4'} immediately. From \eqref{5-1} we also obtain
$$\frac{1}{2}\rho^2(1-\rho^2){v'}^2\leq \alpha(N-2-\alpha)\left( \frac{1}{2}-\frac{(1-2\eps)^2}{2}+\frac{(1-2\eps)^{p+1}}{p+1}-\frac{1}{p+1} \right),$$
which yields \eqref{5-5}.
\end{proof}

\subsection{Asymptotic number of intersections with the singular solution}
\label{SS:intersection}
Let us define
\begin{equation}
 \label{defW}
w(\rho,c)=v(\rho,c)-1=\frac{u(\rho,c)}{u_{\infty}(\rho)}-1,
\end{equation}
with the convention that $w(0,c)=-1$ for all $c$, and
\begin{equation}\label{2-2-1}
R(\rho,c):=\sqrt{w(\rho,c)^{2}+\rho^{2}w^{\prime}(\rho,c)^{2}}, \qquad \forall \,\, 0<\rho<1.
\end{equation}
Then:
\begin{lem}
\label{L:Theta}
There exists a continuous function
$\Theta:[0,1)\times [0,\infty)\to \RR$
such that
\begin{equation*}
 \forall (\rho,c)\in [0,1)\times \RR,\quad \tan \Theta(\rho,c)=\frac{\rho w'(\rho,c)}{w(\rho,c)},
\end{equation*}
(with the convention that $\Theta(\rho,c)\in \frac{\pi}{2}+\pi \ZZ \iff w(\rho,c)=0$) and
\begin{gather*}
 \forall \rho\in [0,1),\quad \forall c\geq 0,\quad \Theta(0,c)=\Theta(\rho,0)=\pi.
\end{gather*}
\end{lem}
\begin{proof}
By uniqueness of solutions of ODEs, a solution starting at $(u,u')=(c,0)$ at $\rho=0$ cannot have $(u(\rho,c),u'(\rho,c))=(u_{\infty}(\rho),u'_{\infty}(\rho))$, which implies that $R(\rho,c)>0$ for $0<\rho<1$. We may thus define
\begin{equation}\label{2-2-2}
\overline{\Theta}(\rho,c):=\arctan\left(\frac{\rho w^{\prime}(\rho,c)}{w(\rho,c)}\right) \in \ZZ/\pi\ZZ,
\end{equation}
with the convention that $\overline{\Theta}(\rho,c)\equiv \pi/2$ is $w(\rho,c)=0$.
By Proposition \ref{P:local0}, $\overline{\Theta}(\rho,c)$ is a continuous function on $\{(\rho,c)\,|\,0<\rho<1,\,c \geq 0\}$. Since
$$\lim_{\rho \to 0^+}w(\rho,c)=-1,\quad \lim_{\rho \to 0^+}\rho w'(\rho,c)=0,$$
uniformly in $c$ on any compact set,
we can extend $\overline{\Theta}$ to a continuous function on $[0,1)\times [0,\infty)$.
Since this region is simply connected, by a standard lifting lemma, we may unambiguously define a real-valued function $\Theta(\rho,c)$ such that $\Theta(\rho,c)\equiv \overline{\Theta}(\rho,c)$ in $\ZZ/\pi \ZZ$,
once we specify its value at some point in the domain. Note that for all $\rho \in [0,1)$, $w(\rho,0)=-1$, $w'(\rho,0)=0$. We thus can set $\Theta(\frac{1}{2},0)=\pi$, which implies $\Theta(\rho,0)=\pi$ for all $\rho\in [0,1)$. Since $\overline{\Theta}(0,c) \equiv 0$ for all $c\geq 0$, this also implies $\Theta(0,c)=\pi$ for all $c\geq 0$.
\end{proof}


The function $\Theta$ counts the number of intersections between $u(\cdot,c)$ and $u_{\infty}$ (see Lemma \ref{L:zeros} below). The following Proposition, which is crucial in the proof of Theorem \ref{Thm0}, implies that this number of intersections goes to infinity as $c\to\infty$.


We recall the definition of the Joseph-Lundgren exponent $p_{JL}$ (see \cite{JL}):
\begin{equation}
\label{defpJL}
p_{JL}:=
\begin{cases}
1+\frac{4}{N-4-2\sqrt{N-1}}, & N\geq11,  \cr +\infty, & 3\leq N\leq10.
\end{cases}
\end{equation}

\begin{prop}\label{prop0}
Assume $N\geq 3$ and  $1+\frac{4}{N-2}<p<p_{JL}$.
Then for any $0<\rho<1$, we have
\begin{equation}\label{2-2-0}
  \lim_{c\rightarrow+\infty}\Theta(\rho,c)=-\infty.
\end{equation}
\end{prop}
\begin{rem}
 The\ednote{Thomas: new remark to clarify the assumption that $\rho$ is small. {\color{red} Wei: Thank you for adding this remark. I have one question: in this Remark 2.7, we should use $\Theta(\rho,c) \in \frac{\pi}{2}\ZZ$ or $\Theta(\rho,c) \in \pi\ZZ+\frac{\pi}{2}$? Which one is better?} {\color{blue} Thomas: if I did not make any mistake, it is true when $\Theta\in \frac{\pi}{2}\ZZ$ so I guess that we can leave it as it is}} limit \eqref{2-2-0} was derived in \cite[Lemma 2]{BMW} for $N=3$, with the additional assumption that $\rho$ is small. Let us mention that this assumption can be easily removed,  since by the equation \eqref{2-2-22} in the proof below,
 $$\Theta(\rho,c) \in \frac{\pi}{2}\ZZ\Longrightarrow \Theta'(\rho,c)<0,$$
 and thus \eqref{2-2-0} for some $\rho_0$ implies the same property for all $\rho\in(\rho_0,1)$.
\end{rem}

\begin{proof}
We fix $\eps>0$ small, and let $M=M\left(\frac{\eps}{2}\right)$ given by Proposition \ref{P:Thomas}. Let
\begin{equation}\label{T48}
  \rho_1(c)=\frac{M}{c^{\frac{p-1}{2}}}.
\end{equation}

As a consequence of Proposition \ref{P:Thomas}, and since, by \eqref{2-2-4}, $H_v$ is nonincreasing, we derive the following upper bound for $\rho\geq \rho_1(c)$, and $c\geq c_0$ (where $c_0$ is a large constant depending only on $\eps$):
\begin{equation}\label{2-2-14}
H_{v}(\rho,c)\leq H_{v}\left(\rho_{1},c\right)\leq -\alpha(N-2-\alpha)\left( \frac{(1-\eps)^2}{2} -\frac{(1-\eps)^{p+1}}{p+1}\right),
\end{equation}
for any $\rho\in[\rho_{1},1]$. On the other hand, we have
\begin{eqnarray}\label{2-2-15}
H_{v}(\rho,c)&=&\frac{1}{2}\rho^{2}\left(1-\rho^{2}\right)(v')^{2}-\alpha(N-2-\alpha)\left(\frac{v^{2}}{2}-\frac{v^{p+1}}{p+1}\right) \\
\nonumber &\geq&-\alpha(N-2-\alpha)\left(\frac{v^{2}}{2}-\frac{v^{p+1}}{p+1}\right).
\end{eqnarray}
Combining \eqref{2-2-14} with \eqref{2-2-15}, we arrive at
\begin{equation}\label{2-2-16}
\alpha(N-2-\alpha)\left(\frac{v^{2}}{2}-\frac{v^{p+1}}{p+1}\right)\geq \alpha(N-2-\alpha)\left(\frac{(1-\eps)^{2}}{2}-\frac{(1-\eps)^{p+1}}{p+1}\right).
\end{equation}
for any $\rho\in[\rho_{1},1]$. Since $v\mapsto \frac{v^2}{2}-\frac{|v|^{p+1}}{p+1}$ is increasing on $[0,1]$, we deduce from \eqref{2-2-16} that
\begin{equation}\label{2-2-20}
  v(\rho,c)>1-\eps,\quad \forall \, c\geq c_{0},\quad \forall \, \rho\in[\rho_{1}(c),1].
\end{equation}

Now, we consider the equation satisfied by the function $\Theta(\rho,c)$: for $0<\rho<1$,
\begin{eqnarray}\label{2-2-22}
\Theta^{\prime}(\rho,c)&=&-\frac{1}{\rho}\bigg[\sin ^{2} \Theta+\frac{(N-2-2\alpha)-\rho^{2}}{1-\rho^{2}} \sin \Theta \cos \Theta \\
\nonumber &&\qquad+\frac{\alpha(N-2-\alpha)v}{1-\rho^{2}}\left(\frac{1-|v|^{p-1}}{1-v}\right) \cos ^{2} \Theta\bigg]\\
\nonumber &=:&-\frac{1}{\rho}\left[A\sin ^{2} \Theta+B\sin \Theta \cos \Theta+C\cos ^{2} \Theta\right],
\end{eqnarray}
with the coefficients
\begin{equation}\label{2-2-23}
A:=1, \quad B:=\frac{(N-2-2\alpha)-\rho^{2}}{1-\rho^{2}}, \quad C:=\frac{\alpha(N-2-\alpha)v}{1-\rho^{2}}\left(\frac{1-|v|^{p-1}}{1-v}\right).
\end{equation}

We will show that, there exists a $0<\rho_{2}<<1$ such that, for any $c\geq c_{0}$ sufficiently large, the quadratic form w.r.t. $\sin\Theta$ and $\cos\Theta$ in the bracket $[\cdots]$ in \eqref{2-2-22} is bounded from below by a positive constant $\delta$ (depending only on $N$ and $p$) for any $\rho_{1}(c)\leq\rho\leq\rho_{2}$, that is,
\begin{equation}\label{2-2-28}
  A\sin ^{2} \Theta+B\sin \Theta \cos \Theta+C\cos ^{2} \Theta\geq\delta>0, \qquad \forall \,\, \rho_{1}(c)\leq\rho\leq\rho_{2}.
\end{equation}

To this end, we only need to show that there exist a $0<\rho_{2}<<1$ and a positive constant $\eta$ (depending only on $N$ and $p$) such that the discriminant $\Delta=B^{2}-4AC$ satisfies $\Delta \leq-\eta$ for any $v>1-\eps$ and $\rho\leq\rho_{2}$. Since $\rho\leq\rho_{2}<<1$, we may replace these coefficients $A$, $B$ and $C$ by
\begin{equation}\label{2-2-24}
A=1, \quad \tilde{B}:=N-2-2\alpha, \quad \tilde{C}:=\alpha(N-2-\alpha)v\left(\frac{1-|v|^{p-1}}{1-v}\right),
\end{equation}
and (using the lower bound \eqref{2-2-20} of $v$), show that the discriminant
\begin{equation}\label{2-2-25}
\tilde{\Delta}=\tilde{B}^{2}-4A\tilde{C}=(N-2-2\alpha)^{2}-4\alpha(N-2-\alpha)\frac{v^{p}-v}{v-1}\leq-2\eta<0
\end{equation}
for any $v>1-\eps$. Since $\tilde{\Delta}$ is a decreasing function of $v$, \eqref{2-2-25} follows immediately, for small $\eps$, once we have derived the following inequality:
\begin{equation}\label{2-2-26}
(N-2-2\alpha)^{2}-4\alpha(N-2-\alpha)(p-1)<0.
\end{equation}
Since $1+\frac{4}{N-2}<p<p_{JL}$, one has $\max\{\frac{N-4}{2}-\sqrt{N-1},0\}<\alpha<\frac{N-2}{2}$, and hence $\left|\alpha-\frac{N-4}{2}\right|<\sqrt{N-1}$. By direct computation, using that $\alpha(p-1)=2$, we have
\begin{equation}\label{T49}
  (N-2-2\alpha)^{2}-4\alpha(N-2-\alpha)(p-1)=4\left[\left(\alpha-\frac{N-4}{2}\right)^{2}-\left(N-1\right)\right]<0,
\end{equation}
which yields \eqref{2-2-26} immediately.

Therefore, we have derived \eqref{2-2-25} and hence the existence of $0<\rho_{2}\ll 1$ such that $\Delta\leq-\eta<0$ for any $v>1-\eps$ and $\rho\leq\rho_{2}$. Thus \eqref{2-2-28} holds for any large $c$ and any $\rho$ such that $\rho_{1}(c)\leq\rho\leq\rho_{2}$.

Consequently, we derive from \eqref{2-2-22} and \eqref{2-2-28} that, for any given $0<\rho<1$,
\begin{eqnarray}\label{2-2-27}
&& \lim _{c \rightarrow +\infty} \Theta\left(\rho, c\right) \\
\nonumber &=&\lim _{c \rightarrow +\infty}\left[\Theta\left(\bar{\rho}(c),c\right)+\left(\int_{\bar{\rho}(c)}^{\rho_{1}(c)}+\int_{\rho_{1}(c)}^{\rho_{2}}+\int_{\rho_{2}}^{\rho}\right)\Theta^{\prime}(\rho,c)d\rho\right] \\
\nonumber &\leq& \frac{3}{2}\pi+\lim _{c \rightarrow +\infty}\left[\int_{\bar{\rho}(c)}^{\rho_{1}(c)}\frac{2(N-2-2\alpha)}{\rho}d\rho+
\int_{\rho_{1}(c)}^{\rho_{2}}\left(-\frac{\delta}{\rho}\right)d\rho+\int_{\rho_{2}}^{\rho}\frac{(N-2-2\alpha)+\rho^{2}}{\rho_{2}(1-\rho^{2})}d\rho\right] \\
\nonumber &\leq&\frac{3}{2}\pi+2(N-2-2\alpha)\ln\frac{M}{b_{\infty}^{\frac{p-1}{2}}}+\frac{N-1-2\alpha}{\rho_{2}(1-\rho^{2})}-\delta\lim _{c \rightarrow+\infty} \ln \frac{\rho_{2}}{\rho_{1}(c)}=-\infty,
\end{eqnarray}
where $\bar{\rho}(c):=\min\{\rho_{1}(c),\rho_{\ast}(c)\}$ and $\rho_{\ast}(c)>\left(\frac{b_{\infty}}{c}\right)^{\frac{p-1}{2}}$ is the first root of $w(\rho,c)$ (if $w(\rho,c)<0$ for $\rho\in(0,1]$, we set $\rho_{\ast}(c)=1$). Since $\lim_{\rho\rightarrow0+}\Theta(\rho,c)=\pi$, one can infer that $\Theta\left(\bar{\rho}(c),c\right)\leq\frac{3}{2}\pi$. This concludes our proof of Proposition \ref{prop0}.
\end{proof}

We conclude this subsection by giving the link between the functions $\Theta$ and the number of intersections between $u(\cdot,c)$ and $u_{\infty}$.
\begin{lem}
\label{L:zeros}
 Let $0\leq \rho_1<\rho_2<1$. Then the number of zeros of $w(\cdot,c)$ on $[\rho_1,\rho_2)$ is given by
 $$\ent{\frac{\Theta(\rho_1,c)}{\pi}-\frac{1}{2}}-\ent{\frac{\Theta(\rho_2,c)}{\pi}-\frac{1}{2}},$$
 where $\ent{x}$ is the integer part of $x$.
\end{lem}
\begin{proof}
By the definition of $\Theta$,
 $$w(\rho,c)=0 \iff \Theta(\rho,c)\in \pi\mathbb{Z}+\frac{\pi}{2}.$$
 We first consider the case where $w(\cdot,c)$ has at least one zero on $[\rho_1,\rho_2)$.
If $\xi$ is one of these zeros, we know that $w'(\xi,c)\neq 0$ and, by \eqref{2-2-22},
$$\Theta'(\xi,c)=-\frac{1}{\xi}<0.$$
Thus $\Theta(\cdot,c)$ decreases in a neighborhood of $\xi$. This proves that if $w(\cdot,c)$ has at least one zero on $[\rho_1,\rho_2)$, one must have $\Theta(\rho_1,c)>\Theta(\rho_2,c)$, and that the number of zeros of $w(\cdot,c)$ on $[\rho_1,\rho_2)$ is given by the cardinal of
$$\big(\Theta(\rho_2,c),\Theta(\rho_1,c)\big]\cap \left\{\pi \mathbb{Z}+\frac{\pi}{2}\right\},$$
which is exactly $\ent{\frac{\Theta(\rho_1,c)}{\pi}-\frac{1}{2}}-\ent{\frac{\Theta(\rho_2,c)}{\pi}-\frac{1}{2}}$.

If $w(\cdot,c)$ has no zero on $[\rho_1,\rho_2)$, it means that
$\{\Theta(\rho,c),\; \rho_1\leq \rho<\rho_2\}$ does not intersect $\frac{\pi}{2}+\pi\mathbb{Z}$. Note that if $w(\rho_2)=0$, then by the above consideration $\Theta(\rho,c)>\Theta(\rho_2,c)$ if $\rho\in [\rho_1,\rho_2)$. In any case, it is easy to check that   $\ent{\frac{\Theta(\rho_1,c)}{\pi}-\frac 12}-\ent{\frac{\Theta(\rho_2,c)}{\pi}-\frac 12}=0$.
 \end{proof}

\section{Existence of self-similar solutions in the case $p<1+\frac{4}{N-3}$}
\label{S:subcritical_case}

In this Section, we will carry out the proof of Theorem \ref{Thm0}. The proof follows the general strategy of the appendix of \cite{BMW} (see also \cite{L}) where the result is proved for $N=3$ and integer $p$. The two crucial ingredients are Propositions \ref{P:Thomas} and \ref{prop0} proved in Section \ref{S:regular}.  We first develop the well-posedness theory for equation \eqref{PDE} at $\rho=1$.  The conclusion of the proof of Theorem \ref{Thm0}, in Subsection \ref{SS:SM}, is divided into two propositions: Proposition \ref{prop1} is devoted to the case where the number of intersections between $u$ and $u_{\infty}$ is even and Proposition \ref{prop2} to the case (which is not detailed in \cite{BMW}) where it is odd.
\subsection{Existence of regular solutions at the boundary of the wave cone}
In the case $p<1+\frac{4}{N-3}$, the well-posedness statement in the neighborhood of $\rho=1$ is as follows.
\begin{prop}
\label{P:local1}
Assume $1+\frac{4}{N-2}<p<1+\frac{4}{N-3}$.
Let $b\in \RR$. There exists a unique $C^2$ solution $U(\cdot,b)$ of \eqref{PDE} defined in a neighborhood of $\rho=1$ such that $U(1,b)=b$. The solution $U$ can be extended to a maximal solution $(0,\rho_+(b))$, with $\rho_+(b)\in (1,\infty]$, such that
$$\rho_+(b)<\infty\Longrightarrow \lim_{\rho\to\rho_+(b)}|U(\rho,b)|=+\infty.$$
Furthermore,
$$U'(1,b)=\frac{b^p-b_0^{p-1}b}{\frac{2(p+1)}{p-1}-(N-1)}.$$
The set $\Omega=\left\{ (\rho,b)\in (0,+\infty)\times \RR,\; 0<\rho<\rho_+(b)\right\}$ is open and $U$ and $U'$ are continuous on $\Omega$.
\end{prop}
Of course, Proposition \ref{P:local1} does not exclude solutions which are regular at $\rho=1$ and singular at $\rho=0$, as shows the explicit singular solution $u_{\infty}$, defined by \eqref{singular}. Singularity might also form for $\rho>1$. More precisely,
 we will prove in Section \ref{S:extension} that $\rho_+(b)<\infty$ if $|b|>b_0$, and that $\rho_+(b)=+\infty$ if $|b|\leq b_0$ and $N=3$, or if $N\geq 4$ and $b_{\infty}-\eps<|b|\leq b_0$ for some $\eps>0$.
\begin{rem}
 The assumption $N=3$ or $p<1+\frac{4}{N-3}$ is crucial in Proposition \ref{P:local1}. See Proposition \ref{P:local2} for the case $N\geq4$, $p=1+\frac{4}{N-3}$. The case $N\geq 4$, $p>1+\frac{4}{N-3}$ seems more complicated, as show the theoretical and numerical investigations in \cite{R}.
\end{rem}

\begin{proof}[Proof of Proposition \ref{P:local1}]
\noindent\emph{Step 1. Local existence and uniqueness.}
We start by proving the local existence and uniqueness for solutions of \eqref{PDE} with initial data at the boundary point $\rho=1$.

We fix $b>0$. The case $b<0$ can be deduced by changing $U$ into $-U$. The case $b=0$ is commented in the end of this step.

We see that $U$ solves \eqref{PDE} if and only if
\begin{equation}
 \label{LWP10}
 \left( \rho^{N-1}\left|\rho^2-1 \right|^{\alpha-\frac{N-1}{2}}(\rho^2-1)U' \right)'=-\rho^{N-1}|\rho^2-1|^{\alpha-\frac{N-1}{2}}U(b_0^{p-1}-|U|^{p-1}).
\end{equation}
Integrating twice, we see that $U$ is a $C^2$ solution of \eqref{PDE} such that $U(1)=b$, on an interval $I$ containing $1$ if and only if
$$\forall \rho\in I,\quad U(\rho)=b+\Psi(U),$$
where
\begin{equation}
 \label{LWP11}
 \Psi(U)(\rho)=\int_{\rho}^1\int_{\tau}^1 \left( \frac{\sigma}{\tau} \right)^{N-1}\left( \frac{1-\sigma^2}{1-\tau^2} \right)^{\alpha-\frac{N-1}{2}}U(\sigma)\left(b_0^{p-1}-|U(\sigma)|^{p-1}\right)d\sigma (1-\tau^2)^{-1}d\tau.
\end{equation}
We note that we have used that $p<1+\frac{4}{N-3}$, so that $\frac{N-3}{2}<\alpha$. This implies that there is no boundary term at $\rho=1$ when integrating \eqref{LWP10}, and that the integral defining $\Psi(U)$ converges if $U$ is continuous.

We fix a small $\eps_b>0$ to be specified. We will prove that $U\mapsto b+\Psi(U)$ is a contraction of the metric space
$$Y_b:=\left\{U\in C^0([1-\eps_b,1+\eps_b],\RR),\; \forall \rho\in [1-\eps_b,1+\eps_b], \; 0\leq U(\rho)\leq 2b\right\},$$
with the metric induced by the $L^{\infty}$ norm, and that
\begin{equation}
 \label{LWP20}
 \forall (U,V)\in Y_b^2,\quad
 \left\|\Psi(U)-\Psi(V)\right\|_{\infty}\leq \frac{1}{4}\|U-V\|_{\infty},
\end{equation}
where $\|U\|_{\infty}=\sup_{1-\eps_b\leq \rho\leq 1+\eps_b} |U(\rho)|$. Indeed, since $p>1$, there exists $k_b>0$ such that
\begin{equation}
\label{LWP21}
\forall (x,y)\in [0,2b],\quad \left|x\left( b_0^{p-1}-x^{p-1} \right)-y\left(b_0^{p-1}-y^{p-1}\right)\right|\leq k_b|x-y|.
\end{equation}
Chosing $\eps_b$ small enough, and assuming $1-\eps_b\leq \rho\leq 1+\eps_b$, we have

$$ \left(\frac{\sigma}{\tau}\right)^{N-1}\left( \frac{1+\sigma}{1+\tau} \right)^{\alpha-\frac{N-1}{2}}\leq 2$$

in the integrand defining $\Psi$ in \eqref{LWP11}. Thus if $U,V\in Y_b$,
\begin{align*}
 \left|\Psi(U)(\rho)-\Psi(V)(\rho)\right|& \leq 2k_b \int_{\rho}^1 \int_{\tau}^1 (1-\sigma)^{\alpha-\frac{N-1}{2}}\|U-V\|_{\infty}d\sigma (1-\tau)^{\frac{N-3}{2}-\alpha}d\tau\\
 &\leq \frac{4 k_b}{2\alpha-N+3}\int_{\rho}^1 \|U-V\|_{\infty}d\tau\\
 &\leq \frac{4 k_b}{2\alpha-N+3}\eps_b\|U-V\|_{\infty}.
\end{align*}
Taking $\eps_b\leq\frac{2\alpha-N+3}{16k_{b}}$, we deduce \eqref{LWP20}. Note that \eqref{LWP20} with $V=0$ implies $\|\Psi(U)\|_{\infty}\leq \frac{b}{2}$ if $U\in Y_b$, so that
\begin{equation}
 \label{LWP30} \frac{b}{2}\leq b+\Psi(U(\rho))\leq \frac{3b}{2},\quad 1-\eps_b\leq \rho\leq 1+\eps_b,
\end{equation}
for $U\in Y_b$. We have proved that $U\mapsto b+\Psi(U)$ is a contraction on $Y_b$, which yields the local existence and uniqueness of $U(\cdot,b)$.

It remains to treat the case $b=0$. The solution is $U(\rho,0)=0$. Uniqueness can be proved showing that $\Psi$ is a contraction close to $b=0$, in the space
$$Y:=\left\{U\in C^0([1-\eps,1+\eps],\RR),\; \forall \rho\in [1-\eps,1+\eps], \; -b_0/2\leq U(\rho)\leq b_0/2\right\},$$
for small enough $\eps$.

\medskip

\noindent\emph{Step 2. Continuity of the flow.}

We next prove the continuity of the flow close to $\rho=1$,  For this we fix $B\in (0,\infty)$, and notice that by the definition \eqref{LWP21} of $k_b$, we can take $k_b$, and thus $\eps_b$ independent of $b$ for $\frac{1}{2}B\leq b\leq \frac 43B$.

Fixing $b\in \left[\frac{1}{2} B,\frac 43B\right]$, we see by \eqref{LWP30} that $U(\cdot,b)$ is in $Y_{B}$. Thus by \eqref{LWP20},
\begin{align*}
\|U(\cdot,b)-U(\cdot,B)\|_{\infty}&\leq \left\|\Psi(U(\cdot,b))-\Psi(U(\cdot,B))\right\|_{\infty}+|b-B|\\
&\leq \frac{1}{4} \left\|U(\cdot,b)-U(\cdot,B)\right\|+|b-B|.
\end{align*}
Thus
\begin{equation}
\label{bound_bB}
\forall \rho\in [1-\eps_{B},1+\eps_{B}],\quad |U(\rho,b)-U(\rho,B)|\leq \frac 43|b-B|.
\end{equation}
Let $\rho_k\to \rho\in [1-\eps_B,1+\eps_B]$ and $b_k\to B$. Then
$$ U(\rho_k,b_k)-U(\rho,B)=U(\rho_k,b_k)-U(\rho_k,B)+U(\rho_k,B)-U(\rho,B).$$
Using \eqref{bound_bB} and the continuity of $U(\cdot,B)$, we obtain that $U$ is continuous in a neighborhood of $(1,B)$ in $(0,\infty)\times \RR$, for $B>0$ and also, by symmetry, for $B<0$. The case $B=0$ can be treated in a similar fashion.

\medskip

\noindent\emph{Step 3. Extension to $(0,1]$.}

 We use the function $H$ defined by \eqref{2-3-0}.
 Denoting $u(\rho)=U(\rho,b)$, there exists a $\delta>0$ small enough such that $H$, $u$ and $u'$ are finite on $1-\delta\leq\rho\leq1$. Moreover, using \eqref{2-3-0} and \eqref{2-3-1}, and that the minimum of the function $F$ is $F(b_0)=-\frac{b_0^2}{p-1}$, one can verify that
\begin{equation}\label{2-3-5}
-\frac{H^{\prime}(\rho)}{1+\frac{1}{p-1}b_{0}^{2}+H(\rho)}\leq\frac{2(N-1)}{\delta\rho}, \qquad \forall \,\, \rho\leq 1-\delta.
\end{equation}

Hence, by integrating from $\rho=1-\delta$ to the left, we obtain that $H$ and thus $u$ and $u'$ stays finite on $0<\rho\leq 1$.

\end{proof}

\subsection{Shooting method}
\label{SS:SM}
We first define the analog of $\Theta(\rho,c)$ for the function $U(\rho,b)$ defined by Proposition \ref{P:local1}. Let,  for $\rho\in (0,1]$, $b\geq 0$
\begin{equation}
\label{3-0}
W(\rho,b):=\frac{U(\rho,b)}{u_{\infty}(\rho)}-1,\quad
\tilde{R}(\rho, b):=\sqrt{W(\rho, b)^{2}+\rho^{2}W^{\prime}(\rho, b)^{2}}.
\end{equation}
\begin{lem}
\label{L:tTheta}
Assume $N=3$ or $N\geq 4$ and $1+\frac{4}{N-2}<p<1+\frac{4}{N-3}$.
 There exists a continuous function $\tilde{\Theta}:(0,1]\times \left([0,\infty)\setminus \{b_{\infty}\}\right)\to \RR$ such that, for all $(\rho,b)$ in the domain of definition of $\tilde{\Theta}$,
 \begin{equation}
  \label{deftTheta}
  \tan \tilde{\Theta}(\rho,b)=\frac{\rho W'(\rho,b)}{W(\rho,b)}
 \end{equation}
 (with the convention that $\tilde{\Theta}(\rho,b)\in \frac{\pi}{2}+\pi\ZZ$ is $W(\rho,b)=0$) and:
 \begin{gather}
  \label{3-2}
  \forall \rho\in (0,1],\quad \tilde{\Theta}(\rho,0)=\pi,\quad \lim_{b\to\infty} \tilde{\Theta}(1,b)=\frac{\pi}{2},\\
  \label{tTheta1} \forall b\in [0,b_{\infty}),\quad\frac{\pi}{2}<\tilde{\Theta}(1,b)<\frac{3\pi}{2} \text{ and }  \forall \rho\in (0,1], \; \tilde{\Theta}(\rho,b)>\frac{\pi}{2},\\
  \label{tTheta2}
  \forall b\in (b_{\infty},+\infty), \quad -\frac{\pi}{2}<\tilde{\Theta}(1,b)<\frac{\pi}{2}\text{ and }
  \forall \rho\in (0,1],\quad \tilde{\Theta}(\rho,b)>-\frac{\pi}{2}.
\end{gather}
\end{lem}
\begin{proof}
By uniqueness of solutions of ODEs,
$$\exists \rho\in (0,1],\quad \left(U(\rho,b),U'(\rho,b)\right)=\left(u_{\infty}(\rho),u'_{\infty}(\rho)\right) \iff b=b_{\infty},$$
which implies that $\tilde{R}(\rho,b)>0$ if $b\neq b_{\infty}$, $\rho\in (0,1]$.
By the same argument as in the proof of Lemma \ref{L:Theta}, we can define $\tilde{\Theta}$ satisfying \eqref{deftTheta} in each of the simply connected regions $\{(\rho,b)\big|\,0<\rho\leq1, \, 0\leq b<b_{\infty}\}$ and $\{(\rho,b)\big|\,0<\rho\leq1, \, b_{\infty}<b<+\infty\}$ once we specify its value at some point of each of these two domains. We have $W'(1,0)/W(1,0)=0$ and (by the formula giving $U'(1,b)$ in Proposition \ref{P:local1})
\begin{equation}
\label{W'(1)}
\frac{W'(1,b)}{W(1,b)}=-\frac{b}{N-3-2\alpha} \frac{b_{\infty}^{p-1}-b^{p-1}}{b_{\infty}-b}\underset{b\to +\infty}{\longrightarrow} +\infty.
\end{equation}
We can thus set:
\begin{equation*}
  \tilde{\Theta}(1,0)=\pi, \quad \lim_{b\rightarrow+\infty}\tilde{\Theta}(1,b)=\frac{\pi}{2},
\end{equation*}
Since $W'(\rho,0)=0$, we obtain that $\tilde{\Theta}(\rho,0)=\pi$ for all $\rho\in (0,1]$. Hence \eqref{3-2}.

Since $W(1,b)=\frac{b}{b_{\infty}}-1\neq 0$ if $0< b<b_{\infty}$, and $\tilde{\Theta}(1,0)=\pi$, we obtain
$\frac{\pi}{2}<\tilde{\Theta}(1,b)<\frac{3\pi}{2}$ if $0\leq b<b_{\infty}$. Furthermore $\tilde{\Theta}'(\rho,b)<0$ when $\tilde{\Theta}(\rho,b)=\frac{\pi}{2}$ (see formula \eqref{2-2-22}, which is also valid for $\tilde{\Theta}$), hence $\tilde{\Theta}(\rho,b)>\frac{\pi}{2}$ for any $0<\rho\leq1$ and $0\leq b<b_{\infty}$. This proves\eqref{tTheta1}.

Using \eqref{W'(1)} and since $\lim_{b\to\infty}\tilde{\Theta}(1,b)=\frac{\pi}{2}$, we see that $\tilde{\Theta}(1,b)<\frac{\pi}{2}$ for $b$ large. Since $W(1,b)\neq 0$ if $b>b_{\infty}$, we deduce $-\frac{\pi}{2}<\tilde{\Theta}(1,b)<\frac{\pi}{2}$ if $b>b_{\infty}$. Using that $\tilde{\Theta}'(\rho,b)<0$ when $\tilde{\Theta}(\rho,b)=-\frac{\pi}{2}$, we also obtain $\tilde{\Theta}(\rho,b)>-\frac{\pi}{2}$ for any $0<\rho\leq1$ and $b>b_{\infty}$, which concludes the proof of \eqref{tTheta2} and of the lemma.
\end{proof}

With exactly the same proof as in Lemma \ref{L:zeros}, if $0<b\neq b_{\infty}$, and $1<\rho_1<\rho_2$, the number of zeros of $W(\cdot,b)$ on $[\rho_1,\rho_2)$ is given by
$$\ent{\frac{\tilde{\Theta}(\rho_1,b)}{\pi}-\frac{1}{2}}-\ent{\frac{\tilde{\Theta}(\rho_2,b)}{\pi}-\frac{1}{2}}.$$

For arbitrary $0<\rho_{0}<1$ close to $1$ (to be determined later), we define a map
\begin{equation}\label{3-4}
\Phi : \, \overline{\mathbb{R}_{+}}=\{c | \, c \geq 0\} \rightarrow \mathbb{R}_{+}^{2}=\left\{(x, y) \in \mathbb{R}^{2} | \, y>0\right\}, \quad \Phi(c):=\left(\Theta\left(\rho_{0}, c\right), R\left(\rho_{0}, c\right)\right),
\end{equation}
and then a map
\begin{equation}\label{3-5}
\Gamma : \, c \in\mathbb{R}\rightarrow \mathbb{R}_{+}^{2}, \quad \Gamma(c):=\left\{\begin{array}{ll}{\Phi(c),} & {\text{ if } \, 0<c<+\infty \, \text { (segment 1)}}, \\ {(\pi, 1-c),} & {\text { if } \, -\infty<c\leq0 \, \text { (segment 2)}}.\end{array}\right.
\end{equation}
If $\gamma$ is a continuous curve from an interval $I$ of $\RR$ to $\RR^2$, we will denote by $\Image[\gamma]=\{\gamma(c),\;c\in I\}$ its image. We let
$$\Image[\Gamma]=\Image_1[\Gamma]\cup\Image_2[\Gamma],$$
where $\Image_1[\Gamma]$ is the subset of $\Image[\Gamma]$ corresponding to segment $1$ and $\Image_2[\Gamma]$ the subset of $\Image[\Gamma]$ corresponding to segment $2$.

It follows from definition \eqref{3-5} and uniqueness of solutions of ODEs that segment 1 and 2 of $\Gamma$ cannot intersect with themselves at any points. Suppose that there exist $c_{1}\in(0,+\infty)$ and $c_{2}\in(-\infty,0]$ such that $\Phi(c_{1})=(\pi,1-c_{2})$. Then $\Theta(\rho,c_1)=\pi$, which implies $w'(\rho,c_1)=0$. As a consequence $1-c_2=R(\rho,c_1)=\left|w(\rho,c_1)\right|$, which yields $\left|w(\rho,c_1)\right|\geq 1$. However by \eqref{2-2-29},
$$-1<w(\rho,c)<\left( \frac{p+1}{2} \right)^{\frac{1}{p-1}}-1<1,$$
which gives a contradiction.

Therefore $\Image_1(\Gamma)\cap \Image_2(\Gamma)=\emptyset$, and $\Gamma$ is a continuous, connected and simple curve in $\mathbb{R}^{2}_{+}$.

Now, for every $n\geq1$, we define a map
\begin{equation}\label{3-6}
\Psi_{n} : \, b\in[0,b_{\infty}) \rightarrow \mathbb{R}_{+}^{2}, \quad \Psi_{n}(b):=\left(\tilde{\Theta}\left(\rho_{0}, b\right)-2n\pi, \tilde{R}\left(\rho_{0}, b\right)\right),
\end{equation}
for the same $\rho_{0}\in(0,1)$ close to $1$ (see Figure \ref{Fig1}).
We have the following proposition.
\begin{prop}\label{prop1}
If $\rho_0$ is close enough to $1$, then for any positive integer $n\geq1$, there exist $0<B_{n}<b_{\infty}$ and $C_{n}>0$ such that
\begin{equation}\label{p1-0}
  \Psi_{n}(B_{n})=\Phi(C_{n}).
\end{equation}
Furthermore, equation \eqref{PDE} possesses a regular solution $u$ defined on $0\leq\rho\leq1$ such that $u(1)=B_{n}$, $u(0)=C_{n}$ and $\frac{u}{u_{\infty}}-1$ has exactly $2n$ zeros.
\end{prop}
\begin{proof}
\noindent\emph{Step 1. Choice of $\rho_0$.}
In this step, we prove that there exists $\rho_0\in (0,1)$, close to $1$, such that
\begin{equation}
 \label{boundW}
 \forall \, b\in (0,b_{\infty}], \quad -1<W(\rho_0,b)<1
\end{equation}
and
\begin{equation}\label{p1-1}
  \forall \,\, n\geq1, \qquad \Image[\Psi_{n}]\cap \Image_2[\Gamma]=\emptyset
\end{equation}
(recall that $\Image_2[\Gamma]=\{\pi\}\times [1,+\infty)$ is the subset of the image of $\Gamma$ corresponding to segment 2).

We start by showing that there exists $\rho_0\in (0,1)$ that satisfies \eqref{boundW}.
Let $b_{\ast}:=\frac{1}{2}b_{\infty}>0$. Note that $U(1,b)=b\geq b_{\ast}>0$ for all $b\in [b_{\ast},b_{\infty}]$. By continuity of $U$, there exists a $\rho_{\ast}\in (0,1)$, close to $1$ such that
$$\forall b\in[b_{\ast},b_{\infty}),\quad \forall \rho\in[\rho_{\ast},1],\quad U(\rho,b)>0 .$$
Since $b_{\ast}<b_{\infty}<b_{0}$ by the assumption $p<1+\frac{4}{N-3}$, there is a $\rho_{0}\in[\rho_{\ast},1)$ such that $U(\rho,b)<b_{0}$ for all $\rho\in[\rho_{0},1]$ and $b\in[0,b_{\ast}]$. Let $b\in (0,b_0)$. By the formula giving $U'(1,b)$ in Proposition \ref{P:local1}, we see that $U'(1,b)<0$. Furthermore by the equation \eqref{PDE} satisfied by $U$,
$$U'(\rho,b)=0\text{ and }0<U(\rho,b)<b_0\Longrightarrow U''(\rho,b)>0.$$
We thus conclude
\begin{equation*}
 \forall b\in (0,b_*],\quad \forall \rho\in [\rho_0,1],\quad U'(\rho,b)<0\text{ and }U(\rho,b)\geq b.
\end{equation*}
As a conclusion, $U(\rho,b)>0$ for any $\rho \in [\rho_0,1]$ and any $b\in (0,b_{\infty}]$, which gives the lower bound for $W$ in \eqref{boundW}. To obtain the upper bound, it suffices to notice that $W(1,b)=b/b_{\infty}-1\leq 0$ for $b\in [0,b_{\infty}]$. By continuity of $W$, taking a $\rho_0$ closer to $1$ if necessary, we obtain \eqref{boundW}.

We next prove \eqref{p1-1}, which is a direct consequence of \eqref{boundW}. We argue by contradiction. Let $(X,Y)\in \Image[\Psi_{n}]\cap \Image_2[\Gamma]$. Then there exists $b\in(0,b_{\infty})$ such that $X=\tilde{\Theta}(\rho_0,b)-2n\pi=\pi$ and $Y=\tilde{R}(\rho_0,b)\geq 1$. Taking the tangent of $X$, we see that $W'(\rho_0,b)=0$. Thus by \eqref{boundW}, $\tilde{R}(\rho_0,b)=|W(\rho_0,b)|<1$, a contradiction.

\medskip

\emph{Step 2. Topological argument.}
Since $\lim_{c\rightarrow+\infty}\Theta(\rho_{0},c)=-\infty$ by Proposition \ref{prop0} and $\Theta(\rho_0,0)=\pi$ by the construction of $\Theta$, there exists a $\bar{c}_{n}>0$ such that $\Theta(\rho_{0},\bar{c}_{n})=\frac{\pi}{2}-2n\pi$ and $\Theta(\rho_{0},c)>\frac{\pi}{2}-2n\pi$ for any $c<\bar{c}_{n}$. Define
$$Z_{n}:=\left\{(x,y)\in\mathbb{R}^{2}_{+}\,\Big |\,x>\frac{\pi}{2}-2n\pi\right\},\quad
\Gamma_{n}:=\Gamma_{\restriction (-\infty,\bar{c}_n]}.$$
Then the curve $\Gamma_{n}$ separates the domain $Z_{n}$ into two open regions: $Z_{n}\setminus \Image[\Gamma_{n}]=X_{n}\cup Y_{n}$ with $X_{n}\cap Y_{n}=\emptyset$. We denote by $X_n$ the region containing $\left(\frac{\pi}{2}-2n\pi,\pi\right)\times [M,+\infty)$ for some large $M$. Thus $Y_n$ is the region with $\partial Y_{n}$ containing $\left(\frac{\pi}{2}-2n\pi,+\infty\right)\times \{0\}$.
We will prove:
\begin{gather}
\label{Step2_1}
\Image[\Psi_{n}]\subset Z_{n},\\
\label{Step2_2}
X_n\cap \Image[\Psi_n]\neq \emptyset,\\
\label{Step2_3}
Y_n\cap \Image[\Psi_n]\neq \emptyset.
\end{gather}
The fact that $\Image[\Psi_n]\subset Z_n$ means that
$$\forall b\in [0,b_{\infty}), \quad \tilde{\Theta}(\rho_0,b)>\frac{\pi}{2},$$
which we have proved when constructing $\tilde{\Theta}$. Hence \eqref{Step2_1}.

To prove \eqref{Step2_2}, we first prove by contradiction that $\Big(\big\{\pi-2n\pi\}\times [1,+\infty)\Big)\cap \Image[\Gamma]=\emptyset$. Indeed, assume that there exists $c_1>0$ such that
$$\Theta(\rho_{0},c_1)=\pi-2n\pi \,\,\, \text{ and } \,\,\, R(\rho_0,c_1)\geq 1.$$
This implies $w'(\rho_0,c_1)=0$ and thus $|w(\rho_0,c_1)|\geq 1$, contradicting the bound \eqref{2-2-29} on $u$. This shows that the curve $\big\{\pi-2n\pi\}\times [1,+\infty)$ is included in $X_n$ or $Y_n$. Since it has non-empty intersection with $X_n$, it must be included in $X_n$. In particular
$$ P_n=(\pi-2n\pi,1)\in X_n.$$
This yields \eqref{Step2_2} since $P_n=\left(\tilde{\Theta}(\rho_0,0)-2n\pi,\tilde{R}(\rho_0,0)\right)\in\Image[\Psi_n]$.

We next prove \eqref{Step2_3}.
Define
$$m_{n}:=\inf_{0\leq c\leq \bar{c}_{n}}R(\rho_{0},c)>0.$$
Since $\tilde{R}(\rho,b_{\infty})=0$ for all $\rho\in (0,1)$, we have $\lim_{b\rightarrow b_{\infty}}\tilde{R}(\rho_{0},b)=0$, and thus there must exist a $\bar{b}_{n}\in(0,b_{\infty})$ such that
$$\tilde{R}(\rho_0,\bar{b}_n)<m_n.$$

As a consequence, the point $Q_{n}:=(\tilde{\Theta}(\rho_{0},\bar{b}_{n})-2n\pi,\tilde{R}(\rho_{0},\bar{b}_{n}))$ is below the curve $\Gamma_n$, which implies $Q_n\in Y_n$. Since
$Q_n\in \Image[\Psi_{n}]$ we are done.

\begin{figure}[htb]
\centering
\includegraphics[width=0.75\textwidth]{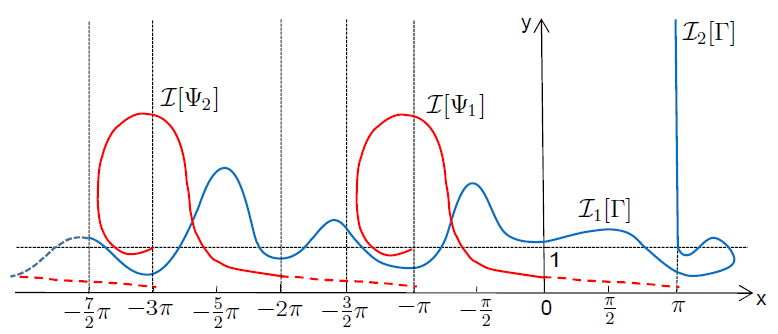}
\caption{Curves for $\Image_{1}[\Gamma]$, $\Image_{2}[\Gamma]$ and $\Image[\Psi_{n}]$ ($n\geq1$).}
\label{Fig1}
\end{figure}

\emph{Step 3. Conclusion of the proof.} Since by Step 2, $\Image[\Psi_n]$ has a point in $X_n$ and a point in $Y_n$, it must intersect $\Image[\Gamma_n]$.
Therefore, $\Image[\Psi_{n}]\cap \Image[\Gamma]\neq\emptyset$, hence by \eqref{p1-1}, there exist $0<B_{n}<b_{\infty}$ and $C_{n}>0$ such that $\Psi_{n}(B_{n})=\Phi(C_{n})$. As a consequence, the functions $u(\rho,C_{n})$ and $U(\rho,B_{n})$ and their derivatives match at $\rho_{0}$. Thus $u(\rho,C_n)=U(\rho,B_n)$ for all $\rho\in(0,1)$, which gives a solution $u$ to \eqref{PDE} defined and regular on the whole interval $0\leq\rho\leq1$. The number of zeros of $\frac{u}{u_{\infty}}-1$ is even and is counted by
\begin{equation}\label{p1-2}
\ent{\frac{\Theta(0,C_n)}{\pi}-\frac{1}{2}}-
\ent{\frac{\Theta(\rho_0,C_n)}{\pi}-\frac{1}{2}}+\ent{\frac{\tilde{\Theta}(\rho_0,B_n)}{\pi}-\frac{1}{2}}-\ent{\frac{\tilde{\Theta}(1,B_n)}{\pi}-\frac{1}{2}}\\
\end{equation}
where $\ent{x}$ denotes the integer part of $x$. Since $\Theta(0,C_n)=\pi$, the contribution of the first term is $0$. Using that $\Psi_n(B_n)=\Phi(C_n)$ we obtain that $\tilde{\Theta}(\rho_0,B_n)-2n\pi=\Theta(\rho_{0},C_n)$, and thus the contribution of the sum of the second and third terms is $2n$. Finally by \eqref{tTheta1}, $\frac{\pi}{2}<\tilde{\Theta}(1,B_n)<\frac{3\pi}{2}$, and thus the contribution of the last term is $0$. Hence the number of zeros of $\frac{u}{u_{\infty}}-1$ is $2n$, as announced.
\end{proof}

Next, for every $n\geq0$, we define a map
\begin{equation}\label{3-7}
\tilde{\Psi}_{n} : \, b\in(b_{\infty},+\infty) \rightarrow \mathbb{R}_{+}^{2}, \quad \tilde{\Psi}_{n}(b):=\left(\tilde{\Theta}\left(\rho_{0}, b\right)-2n\pi, \tilde{R}\left(\rho_{0}, b\right)\right),
\end{equation}
where $\rho_{0}\in(0,1)$ is close to $1$ (to be determined later).

The following proposition concerns the case where the number of intersection $u$ and $u_{\infty}$ is odd.

\begin{prop}\label{prop2}
If $\rho_0$ is close enough to $1$, then for any non-negative integer $n\geq0$, there exist $b_{\infty}<\tilde{B}_{n}<+\infty$ and $\tilde{C}_{n}>0$ such that
\begin{equation}\label{p2-0}
  \tilde{\Psi}_{n}(\tilde{B}_{n})=\Phi(\tilde{C}_{n}).
\end{equation}
Furthermore, equation \eqref{PDE} possesses a regular solution $u$ defined on $0\leq\rho\leq1$ such that $u(1)=\tilde{B}_{n}$, $u(0)=\tilde{C}_{n}$ and $\frac{u}{u_{\infty}}-1$ has exactly $2n+1$ zeros.
\end{prop}
\begin{proof}
\noindent\emph{Step 1. Choice of $\rho_0$.}
In this step, we prove that there exist a large $\hat{b}\in\left(\frac{3}{2}b_{\infty},+\infty\right)$, $\rho_0\in (0,1)$ close to $1$ and $C_{N,p}>0$ such that
\begin{gather}
  \label{p2-4'}
  \forall \, b\in[\hat{b},+\infty), \qquad \tilde{R}(\rho_{0},b)\geq C_{N,p}\,b>\sqrt{2},\\
\label{p2-3'}
\forall b\in \Big[\frac{3}{2}b_{\infty}, \hat{b}\Big], \qquad -\frac{\pi}{2}<\tilde{\Theta}(\rho_{0},b)<\frac{\pi}{2},\\
\label{boundW'}
 \forall \, b\in\Big[b_{\infty},\frac{3}{2}b_{\infty}\Big), \quad -1<W(\rho_0,b)<1
\end{gather}
and
\begin{equation}\label{p1-1'}
  \forall \,\, n\geq0, \qquad
  \tilde{\Psi}_n\big((b_{\infty},\overline{b})\big)\cap \Image_2[\Gamma]=\emptyset,
\end{equation}
where $\bar{b}:=\sup\left\{b\in[\hat{b},+\infty)\,\big|\,\tilde{\Theta}(\rho_{0},\tau)<\frac{3}{4}\pi, \, \forall \, \tau\in[\hat{b},b)\right\}\in(\hat{b},+\infty]$.

Indeed, it can be deduced from \eqref{2-3-0} and \eqref{2-3-1} that, for any $0<\rho\leq1$ and $b\in(0,+\infty)$,
\begin{equation}\label{p2-6}
  \frac{1-\rho^{2}}{2}\left(U'(\rho,b)\right)^{2}+\frac{|U(\rho,b)|^{p+1}}{p+1}-\frac{p+1}{(p-1)^{2}}U^{2}(\rho,b)\geq\frac{b^{p+1}}{p+1}-\frac{p+1}{(p-1)^{2}}b^{2}.
\end{equation}
Thus there exists a $\tilde{\rho}_{\ast}$ close enough to $1$ such that, for any $\rho\in[\tilde{\rho}_{\ast},1]$ and $b\in(0,+\infty)$,
\begin{eqnarray}\label{p2-7}
&& \frac{b^{p+1}}{p+1}-\frac{p+1}{(p-1)^{2}}b^{2} \\
 \nonumber &\leq& \left(1-\rho^{2}\right)\left(\frac{\alpha U}{\rho}+U'\right)^{2}+\frac{1-\rho^{2}}{\rho^{2}}\alpha^{2}U^{2}+\frac{|U(\rho,b)|^{p+1}}{p+1}-\frac{p+1}{(p-1)^{2}}U^{2}(\rho,b) \\
 \nonumber &\leq& \left|\frac{\alpha U(\rho,b)}{\rho}+U'(\rho,b)\right|^{2}+\frac{|U(\rho,b)|^{p+1}}{p+1} \\
 \nonumber &\leq& \left(|U|+\left|\frac{\alpha U}{\rho}+U'\right|\right)^{2}+\frac{1}{p+1}\left(|U|+\left|\frac{\alpha U}{\rho}+U'\right|\right)^{p+1}.
\end{eqnarray}
It follows from \eqref{p2-7} that, for any $\rho\in[\tilde{\rho}_{\ast},1]$,
\begin{equation}\label{p2-4}
  \lim_{b\rightarrow+\infty}\left(|U(\rho,b)|+\left|\frac{\alpha U(\rho,b)}{\rho}+U'(\rho,b)\right|\right)=+\infty,
\end{equation}
and moreover, there exists a $\hat{b}\in\left(\frac{3}{2}b_{\infty},+\infty\right)$ sufficiently large, such that, for any $b\geq\hat{b}$ and $\rho\in[\tilde{\rho}_{\ast},1]$,
\begin{equation}\label{p2-8}
  \frac{b^{p+1}}{2(p+1)}\leq \frac{p+2}{p+1}\left(|U(\rho,b)|+\left|\frac{\alpha U(\rho,b)}{\rho}+U'(\rho,b)\right|\right)^{p+1},
\end{equation}
and hence
\begin{equation}\label{p2-9}
  \forall \, b\in[\hat{b},+\infty), \quad \forall \, \rho\in[\tilde{\rho}_{\ast},1], \qquad |U(\rho,b)|+\left|\frac{\alpha U(\rho,b)}{\rho}+U'(\rho,b)\right|\geq C_{p}b.
\end{equation}

Recalling the definition \eqref{3-0} of $\tilde{R}(\rho,b)$, we have, for any $\rho\in[\tilde{\rho}_{\ast},1]$ (choose $\tilde{\rho}_{\ast}$ closer to $1$ if necessary) and $b\in(0,+\infty)$,
\begin{eqnarray}\label{p2-10}
  \tilde{R}(\rho,b)&=& \sqrt{W(\rho,b)^{2}+\rho^{2}W'(\rho,b)^{2}}=\sqrt{\left(\frac{U(\rho,b)}{u_{\infty}(\rho)}-1\right)^{2}
  +\rho^{2}\left[\left(\frac{U(\rho,b)}{u_{\infty}(\rho)}\right)'\right]^{2}} \\
 \nonumber &=&\frac{1}{b_{\infty}}\sqrt{\left(\rho^{\alpha}U(\rho,b)-b_{\infty}\right)^{2}+\rho^{2\alpha+2}\left(\frac{\alpha U(\rho,b)}{\rho}+U'(\rho,b)\right)^{2}} \\
 \nonumber &\geq&\frac{1}{\sqrt{2}b_{\infty}}\left(\rho^{\alpha}|U(\rho,b)|-b_{\infty}+\rho^{\alpha+1}\left|\frac{\alpha U(\rho,b)}{\rho}+U'(\rho,b)\right|\right) \\
 \nonumber &\geq& \frac{1}{2\sqrt{2}b_{\infty}}\left(|U(\rho,b)|+\left|\frac{\alpha U(\rho,b)}{\rho}+U'(\rho,b)\right|\right)-\frac{1}{\sqrt{2}}.
\end{eqnarray}
As a consequence of \eqref{p2-4} and \eqref{p2-9}, \eqref{p2-10} implies further that, for any $\rho\in[\tilde{\rho}_{\ast},1]$,
\begin{equation}\label{p2-5}
  \lim_{b\rightarrow+\infty}\tilde{R}(\rho,b)=+\infty,
\end{equation}
and furthermore, for any $b\in[\hat{b},+\infty)$ (choose $\hat{b}$ larger if necessary) and any $\rho\in[\tilde{\rho}_{\ast},1]$,
\begin{equation}\label{p2-11}
  \tilde{R}(\rho,b)\geq C_{N,p}b>\sqrt{2}.
\end{equation}

By Lemma \ref{L:tTheta}, $-\frac{\pi}{2}<\tilde{\Theta}(1,b)<\frac{\pi}{2}$ for any $b_{\infty}<b<+\infty$, thus there exists a $\bar{\rho}_{\ast}\in[\tilde{\rho}_{\ast},1)$ close to $1$ such that, for any $\rho\in[\bar{\rho}_{\ast},1]$,
\begin{equation}\label{p2-3}
  \forall \,\, \frac{3}{2}{b}_{\infty}\leq b\leq\hat{b}, \qquad -\frac{\pi}{2}<\tilde{\Theta}(\rho,b)<\frac{\pi}{2}.
\end{equation}
One can also infer that, there exists a $\rho_{0}\in[\bar{\rho}_{\ast},1)$ close enough to $1$ such that $U(\rho,b)>0$ for any $\rho\in[\rho_{0},1]$ and $b\in[b_{\infty},\frac{3}{2}b_{\infty})$. It follows that \eqref{p2-5}, \eqref{p2-11} and \eqref{p2-3} hold at $\rho=\rho_{0}$ (so we have derived
\eqref{p2-4'} and \eqref{p2-3'}), and $U(\rho_{0},b)>0$ and hence $W(\rho_{0},b)>-1$ for any $b_{\infty}\leq b<\frac{3}{2}b_{\infty}$. To obtain the upper bound in \eqref{boundW'}, it suffices to notice that $W(1,b)=b/b_{\infty}-1<\frac{1}{2}$ for $b\in [b_{\infty},\frac{3}{2}b_{\infty})$. By continuity of $W$, taking $\rho_0$ closer to $1$ if necessary, we obtain \eqref{boundW'}.

Next, define $\bar{b}:=\sup\left\{b\in[\hat{b},+\infty)\,\big|\,\tilde{\Theta}(\rho_{0},\tau)<\frac{3}{4}\pi, \,\forall\, \tau\in[\hat{b},b)\right\}\in(\hat{b},+\infty]$, we will prove \eqref{p1-1'} for any $n\geq0$. We argue by contradiction. Suppose $$(X,Y)\in\tilde{\Psi}_n\big((b_{\infty},\bar{b})\big)\cap \Image_2[\Gamma].$$
Recall that $I_2(\Gamma)=\{\pi\}\times [1,+\infty)$. Then, from \eqref{p2-3'} and the definition of $\bar{b}$, we infer that, there must exist $b\in(b_{\infty},\frac{3}{2}b_{\infty})$ such that $X=\tilde{\Theta}(\rho_0,b)-2n\pi=\pi$ and $Y=\tilde{R}(\rho_0,b)\geq 1$. Taking the tangent of $X$, we see that $W'(\rho_0,b)=0$. Thus by \eqref{boundW'}, $\tilde{R}(\rho_0,b)=|W(\rho_0,b)|<1$, a contradiction.

\medskip

\noindent\emph{Step 2. Topological argument.}
Since $\lim_{c\rightarrow+\infty}\Theta(\rho_{0},c)=-\infty$ by Proposition \ref{prop0} and $\Theta(\rho_0,0)=\pi$ by the construction of $\Theta$, there exists a $\tilde{c}_{n}>0$ such that $\Theta(\rho_{0},\tilde{c}_{n})=-\frac{\pi}{2}-2n\pi$ and $\Theta(\rho_{0},c)>-\frac{\pi}{2}-2n\pi$ for any $c\in[0,\tilde{c}_n)$. For every $n\geq0$, define
$$\tilde{Z}_{n}:=\{(x,y)\in\mathbb{R}^{2}_{+}|\,x>-\frac{\pi}{2}-2n\pi\}, \quad \tilde{\Gamma}_{n}:=\Gamma_{\restriction (-\infty,\tilde{c}_n]}.$$
Recall that by construction $\tilde{\Theta}(\rho,b)>-\frac{\pi}{2}$ for any $0<\rho\leq1$ and $b_{\infty}<b<+\infty$, thus $\Image[\tilde{\Psi}_{n}]\subset\tilde{Z}_{n}$. Moreover, the curve $\tilde{\Gamma}_{n}$ separates the domain $\tilde{Z}_{n}$ into two open regions (see Figure \ref{Fig2}): $\tilde{Z}_{n}\setminus\Image[\tilde{\Gamma}_{n}]=\tilde{X}_{n}\cup\tilde{Y_{n}}$ with $\tilde{X}_{n}\cap\tilde{Y_{n}}=\emptyset$. We denote by $\tilde{X}_n$ the region containing $\left(-\frac{\pi}{2}-2n\pi,\pi\right)\times [M,+\infty)$ for some large $M$. Thus $\tilde{Y}_n$ is the region with $\partial\tilde{Y}_{n}$ containing $\left(-\frac{\pi}{2}-2n\pi,+\infty\right)\times \{0\}$. We will prove: for any $n\geq0$,
\begin{gather}
\label{Step2_2'}
\exists \hat{b}_n\in \Big[\frac{3}{2}b_{\infty},\bar{b}\Big),\quad \tilde{\Psi}_n(\hat{b}_n)\in \tilde{X}_n, \\
\label{Step2_3'}
\exists \tilde{b}_n\in \Big(b_{\infty},\frac{3}{2}b_{\infty}\Big), \quad \tilde{\Psi}_n(\tilde{b}_n)\in \tilde{Y}_n.
\end{gather}

To this end, we will first show by contradiction that, for any integer $k\in \ZZ$,
\begin{equation}\label{p2-12}
   \left[k\pi-\frac{\pi}{4},k\pi+\frac{\pi}{4}\right]\times\left(\sqrt{2},+\infty\right)\cap\Image_{1}[\Gamma]=\emptyset.
\end{equation}
Indeed, assume that there exists $c_1>0$ such that
$$\Theta(\rho_{0},c_1)\in \left[k\pi-\frac{\pi}{4},k\pi+\frac{\pi}{4}\right] \,\,\, \text{ and } \,\,\, \sqrt{2}<R(\rho_0,c_1)<+\infty.$$
This implies $\rho_{0}|w'(\rho_0,c_1)|\leq |w(\rho_{0},c_{1})|$ and thus $\sqrt{2}<R(\rho_0,c_1)\leq\sqrt{2}|w(\rho_0,c_1)|$, hence $|w(\rho_0,c_1)|>1$, contradicting the bound \eqref{2-2-29} on $u$.

Now, define $M_{n}:=\sup_{0\leq c\leq \tilde{c}_{n}}R(\rho_{0},c)$ and $m_{n}:=\inf_{0\leq c\leq \tilde{c}_{n}}R(\rho_{0},c)$.

We carry out the proof of \eqref{Step2_2'} by discussing two different cases.

\emph{Case i)} $\bar{b}=+\infty$. By the definition of $\bar{b}$ and \eqref{p2-3'}, in such case, we have
\begin{equation}\label{p2-13}
  \forall \,\frac{3}{2}b_{\infty}\leq b<+\infty, \qquad \tilde{\Theta}(\rho_{0},b)\in\left(-\frac{\pi}{2},\frac{3}{4}\pi\right).
\end{equation}
By \eqref{p2-4'}, there is a $\hat{b}_{n}\in[\frac{3}{2}b_{\infty},+\infty)$ sufficiently large such that $\tilde{R}(\rho_{0},\hat{b}_{n})>M_{n}$. It follows from \eqref{p2-13} that $-\frac{\pi}{2}<\tilde{\Theta}(\rho_{0},\hat{b}_{n})<\frac{3}{4}\pi$, thus the point $\tilde{P}_{n}=(\tilde{\Theta}(\rho_{0},\hat{b}_{n})-2n\pi,\tilde{R}(\rho_{0},\hat{b}_{n}))\in\Image[\tilde{\Psi}_{n}]$ is above the curve of $\tilde{\Gamma}_{n}$, i.e., $\tilde{P}_{n}\in \tilde{X}_{n}$, and hence \eqref{Step2_2'} holds.

\emph{Case ii)} $\bar{b}<+\infty$. By the definition of $\bar{b}$, in such case, we have
\begin{equation}\label{p2-14}
  \tilde{\Theta}(\rho_{0},\bar{b})=\frac{3}{4}\pi.
\end{equation}
Since $\bar{b}>\hat{b}$, \eqref{p2-4'} yields that $\tilde{R}(\rho_{0},\bar{b})>\sqrt{2}$. Let $\hat{b}_{n}:=\bar{b}$ and $$\tilde{P}_{n}:=\left(\tilde{\Theta}(\rho_{0},\hat{b}_{n})-2n\pi,\tilde{R}(\rho_{0},\hat{b}_{n})\right)=\left(\frac{3}{4}\pi-2n\pi,\tilde{R}(\rho_{0},\bar{b})\right).$$ Then $\tilde{P}_{n} \in\Image[\tilde{\Psi}_{n}]\cap\left[\pi-\frac{\pi}{4}-2n\pi,\pi+\frac{\pi}{4}-2n\pi\right]\times\left(\sqrt{2},+\infty\right)$, and it follows from \eqref{p2-12} that $\tilde{P}_{n}\notin\Image[\tilde{\Gamma}_{n}]$ is above the curve of $\tilde{\Gamma}_{n}$, i.e., $\tilde{P}_{n}\in \tilde{X}_{n}$, and hence \eqref{Step2_2'} holds.

Since $\lim_{b\rightarrow b_{\infty}}\tilde{R}(\rho_{0},b)=0$, there must exist a $\tilde{b}_{n}\in(b_{\infty},\frac{3}{2}b_{\infty})$ such that the point $\tilde{Q}_{n}:=(\tilde{\Theta}(\rho_{0},\tilde{b}_{n})-2n\pi,\tilde{R}(\rho_{0},\tilde{b}_{n}))\in\Image[\tilde{\Psi}_{n}]$ and $\tilde{R}(\rho_{0},\tilde{b}_{n})<m_{n}$, hence $\tilde{Q}_{n}$ is below the curve of $\tilde{\Gamma}_{n}$, i.e., $\tilde{Q}_{n}\in \tilde{Y}_{n}$. This establishes \eqref{Step2_3'}.

\begin{figure}[htb]
\centering
\includegraphics[width=0.75\textwidth]{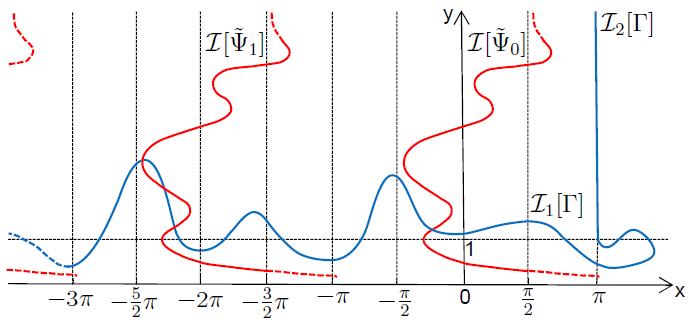}
\caption{Curves for $\Image_{1}[\Gamma]$, $\Image_{2}[\Gamma]$ and $\Image[\tilde{\Psi}_{n}]$ ($n\geq0$).}
\label{Fig2}
\end{figure}

\noindent\emph{Step 3. Conclusion of the proof.} Since by Step 2, $\Image[\tilde{\Psi}_n]$ has a point in $\tilde{X}_n$ and a point in $\tilde{Y}_n$, it must intersect $\Image[\Gamma]$. More precisely, from \eqref{Step2_2'} and \eqref{Step2_3'} in Step 2, we actually have, for any $n\geq0$,
\begin{equation}\label{p2-15}
 \tilde{\Psi}_n\big((b_{\infty},\bar{b})\big)\cap\Image[\Gamma]\neq\emptyset.
\end{equation}
Hence by \eqref{p1-1'}, we have $\tilde{\Psi}_n\big((b_{\infty},\bar{b})\big)\cap\Image_{1}[\Gamma]\neq\emptyset$, and there exist $b_{\infty}<\tilde{B}_{n}<+\infty$ and $\tilde{C}_{n}>0$ such that $\tilde{\Psi}_{n}(\tilde{B}_{n})=\Phi(\tilde{C}_{n})$. As a consequence, the functions $u(\rho,\tilde{C}_{n})$ and $U(\rho,\tilde{B}_{n})$ and their derivatives match at $\rho_{0}$, hence we derive a solution $u$ to \eqref{PDE} defined and regular on the whole interval $0\leq\rho\leq1$. The number of zeros of $\frac{u}{u_{\infty}}-1$ is odd and is counted by
\begin{equation}\label{p2-2}
\ent{\frac{\Theta(0,\tilde{C}_n)}{\pi}-\frac{1}{2}}-\ent{\frac{\Theta(\rho_0,\tilde{C}_n)}{\pi}-\frac{1}{2}}+\ent{\frac{\tilde{\Theta}(\rho_0,\tilde{B}_n)}{\pi}-\frac{1}{2}}-\ent{\frac{\tilde{\Theta}(1,\tilde{B}_n)}{\pi}-\frac{1}{2}}.
\end{equation}
Since  $\Theta(0,\tilde{C}_n)=\pi$,  the contribution of the first term is $0$. Using that $\tilde{\Psi}_n(\tilde{B}_n)=\Phi(\tilde{C}_n)$, we obtain that $\tilde{\Theta}(\rho_0,\tilde{B}_n)-2n\pi=\Theta(\rho_{0},\tilde{C}_n)$, and thus the contribution of the sum of the second and third terms is $2n$. Finally, by \eqref{tTheta2}, $-\frac{\pi}{2}<\tilde{\Theta}(1,\tilde{B}_n)<\frac{\pi}{2}$, and thus the contribution of the last term is $1$. Hence the number of zeros of $\frac{u}{u_{\infty}}-1$ is exactly $2n+1$, as announced. This concludes our proof of Proposition \ref{prop2}.
\end{proof}

\emph{Proof of Theorem \ref{Thm0} (completed).}
Set for any $n\geq 1$, $b_{2n-1}=B_{n}$ and $b_{2n}=\tilde{B}_{n}$, where $B_n$ is given by Proposition \ref{prop1} and $\tilde{B}_n$ by Proposition \ref{prop2}.
Theorem \ref{Thm0} is an immediate consequence of these two propositions. 

\section{Existence of self-similar solutions in the case $p=1+\frac{4}{N-3}$}

\label{S:critical_case}
In this section we prove Theorem \ref{T:critical_case}.
In all this section, we assume $p=1+\frac{4}{N-3}$.

\subsection{Cauchy theory at the boundary of the wave cone}
When $p=1+\frac{4}{N-3}$, the coefficient in front of $u'$ in \eqref{PDE} vanishes at $\rho=1$, and it is thus clear that a regular solution defined in a neighborhood of $\rho=1$ must satisfy $U(1)\in \{0,b_0,-b_0\}$. In the case $U(1)=b_0$, the local well-posedness close to $\rho=1$ takes the following form:
\begin{prop}
 \label{P:local2}
 Assume $p=1+\frac{4}{N-3}$. Then for all $a\in \RR$, there exists a unique $C^2$ solution $U(\rho)=U(\rho,a)$ of \eqref{PDE}, defined in a neighborhood of $\rho=1$, and such that
 \begin{equation}
  \label{BC_local2}
  U(1,a)=b_0,\quad U'(1,a)=a.
 \end{equation}
 Furthermore $U''(1)=-\frac{N-1}{2}a$ and $U$ can be extended to a $C^2$ solution of \eqref{PDE} on $(0,1)$.
\end{prop}
The proof of Proposition \ref{P:local2} yields the following uniqueness statement:
\begin{lem}
 \label{L:local_uniq}
 Let $\tau\in (0,1)$, $a\in \RR$ and $u$ be a $C^2$ solution of \eqref{PDE} on $(\tau,1)$ such that
 \begin{equation}
 \label{limite}
 \lim_{\rho \overset{<}{\to} 1}u(\rho)=b_0,\quad \lim_{\rho \overset{<}{\to} 1}u'(\rho)=a.
 \end{equation}
 Then for all $\rho\in (\tau,1)$, $u(\rho)=U(\rho,a)$.
\end{lem}
\begin{proof}[Proof of Proposition \ref{P:local2}]

 When $\alpha=1+\frac{4}{N-3}$, the equation \eqref{PDE} has the following self-adjoint form
 \begin{equation}
\label{C11}
\left( \rho^{N-1}U' \right)'=\rho^{N-1}(1-\rho^2)^{-1} U\left( b_0^{p-1}-|U|^{p-1} \right).
 \end{equation}
 Let $\eps>0$ be a small parameter (to be specified later). If $U\in C^2([1-\eps,1+\eps])$ satisfies $U(1)=b_0$, we see that the right-hand side of \eqref{C11} is continuous on $[1-\eps,1+\eps]$ and we obtain, integrating again, and denoting $a=U'(1)$,
 \begin{equation}
  \label{C12}
U(\rho)=b_0+\frac{\rho^{2-N}-1}{2-N} a +\int_1^{\rho}\int_{1}^{\tau} \left( \frac{\sigma}{\tau} \right)^{N-1} \left( 1-\sigma^2 \right)^{-1} U(\sigma)\left( b_0^{p-1}-|U(\sigma)|^{p-1} \right)d\sigma d\tau.
 \end{equation}
 Let
 $$Y_a=\Big\{ U\in C^{1}\left( [1-\eps,1+\eps]\right)\,\big|\, U(1)=b_0,\; U'(1)=a\text{ and }\forall \rho\in [1-\eps,1+\eps],\; \left|U'(\rho)-a\right|\leq 1\Big\}.$$
 Defining $d(U,V)=\max_{|\rho-1|\leq \eps}\left|U'(\rho)-V'(\rho)\right|$, we see that $(Y_a,d)$ is a complete metric space, and that
\begin{equation}
\label{nice_Ya}
 \forall (U,V)\in Y_a^2,\; \forall \rho\in [1-\eps,1+\eps], \quad |U(\rho)-V(\rho)|\leq |1-\rho|\,d(U,V).
\end{equation}
Let, for $U\in Y_a$,
$$
\Phi(U)(\rho) =\int_1^{\rho}\int_{1}^{\tau} \left( \frac{\sigma}{\tau} \right)^{N-1} \left( 1-\sigma^2 \right)^{-1} U(\sigma)\left( b_0^{p-1}-|U(\sigma)|^{p-1} \right)d\sigma d\tau.
 $$
There exists a constant $k$, depending only on $N$, such that
\begin{equation}
 \label{lipschtiz}
\forall (X,Y)\in [b_0-1,b_0+1]^{2},\quad
\left| X\left(b_0^{p-1}-|X|^{p-1}\right)-Y\left(b_0^{p-1}-|Y|^{p-1}\right)\right| \leq k|X-Y|.
\end{equation}
If $U\in Y_a$, then $|U(\rho)-b_0|\leq (1+|a|)|\rho-1|\leq \eps(1+|a|)$ for $\rho\in [1-\eps,1+\eps]$. Choosing $\eps$ small enough, we obtain, for $(U,V)\in Y_a^2$,
$$ \left|U(\sigma)\left( b_0^{p-1}-|U(\sigma)|^{p-1} \right)-V(\sigma)\left( b_0^{p-1}-|V(\sigma)|^{p-1} \right)\right|\leq k|U(\sigma)-V(\sigma)|\leq k|\sigma-1|d(U,V).$$
Thus, for $1-\eps\leq \rho\leq 1+\eps$,
\begin{equation*}
 \left|\frac{d}{d\rho}\Big( \Phi(U)(\rho)-\Phi(V)(\rho) \Big)\right|\leq 2k\left|\int_{1}^{\rho} \left( \frac{\sigma}{\rho} \right)^{N-1} d\sigma\right| d(U,V)\leq 4k\eps d(U,V).
\end{equation*}
Letting $\eps$ small enough, we obtain, for $1-\eps\leq \rho\leq 1+\eps$,
$$ \left|\frac{d}{d\rho} (\Phi(U)(\rho)-\Phi(V)(\rho))\right|\leq \frac{1}{10} d(U,V).$$
Thus $U\mapsto b_0+\frac{\rho^{2-N}-1}{2-N}a+\Phi(U)$ is a contraction on $Y_a$. This yields a fixed point for this mapping. Using the definition of $\Phi$, we see that this solution is indeed $C^2$ on $[1-\eps,1+\eps]$ and satisfies \eqref{PDE} on this interval.
Using a Taylor expansion of order $1$ of the right-hand side of \eqref{C11}, we deduce $U''(1)=-\frac{N-1}{2}a$.

The proof of the fact that $U$ can be extended to a solution of \eqref{PDE} on $(0,1]$ is the same as in the case $p<1+\frac{4}{N-3}$ and we omit it.
\end{proof}
\begin{proof}[Proof of Lemma \ref{L:local_uniq}]
 We note that if $\eps>0$ is small enough, the preceding fixed point argument works exactly the same in the complete metric space
 $$Y_a^-=\left\{ U\in C^{1}\left( [1-\eps,1]\right)\,\big|\, U(1)=b_0,\; U'(1)=a\text{ and }\forall \rho\in [1-\eps,1],\; \left|U'(\rho)-a\right|\leq 1\right\},$$
 with the distance $d_-(U,V)=\max_{1-\eps\leq \rho \leq 1}\left|U'(\rho)-V'(\rho)\right|$. If $u$ is a solution of \eqref{PDE} defined on $(\tau,1)$ (where $\tau<1$) and satisfying \eqref{limite}, then for $\eps>0$ small enough, $u_{\restriction [1-\eps,1)}$ can be extended to a $C^1$ function on $[1-\eps,1]$ which is in the space $Y_a$. Uniqueness in the fixed point then yields that $u(\rho)=U(\rho,a)$ for $1-\eps<\rho<1$. This concludes the proof of the lemma.
\end{proof}

We finish this section by a sufficient condition of regularity at $\rho=1$, for the solution $u(\rho,c)$ defined in Proposition \ref{P:local0} when $p=1+\frac{4}{N-3}$.
\begin{lem}
 \label{L:C40}
 Assume $p=1+\frac{4}{N-3}$. Let $u(\rho,c)$ be the solution defined in Proposition \ref{P:local0}. Assume $u(1,c)=b_0$. Then $u(\rho,c)$ can be extended to a $C^2$ solution of \eqref{PDE} close to $\rho=1$. In other terms,
 $$a=\lim_{\rho\overset{<}{\to} 1}u'(\rho,c)$$
 exists in $\RR$ and  $u(\rho,c)=U(\rho,a)$.
\end{lem}
\begin{proof}
Assume to fix ideas that $c>0$. Let $c_0>0$ be a large constant. Since by Proposition \ref{P:local0} $u$ is continuous on $[0,1]\times [0,c_0]$, is also bounded in this set. By the equation
\begin{equation}
\label{C11'}
\left( \rho^{N-1}u' \right)'=\rho^{N-1}(1-\rho^2)^{-1} u\left( b_0^{p-1}-|u|^{p-1} \right),
\end{equation}
we see that there exists a constant $M$, depending only on $c_0$, such that
$$ \left| \left( \rho^{N-1}u'(\rho,c) \right)'\right|\leq \frac{M}{1-\rho},\text{ for }\rho\in (0,1), c\in [0,c_0].$$
Integrating, we deduce that there exists a constant $M'=M'(c_0)$ such that
\begin{equation}
 \label{C50}
 \forall \rho \in \Big[\frac 12,1\Big), \; \forall c\in [0,c_0],\quad |u'(\rho,c)|\leq M' |\log (1-\rho)|.
\end{equation}
Assume that $u(1,c)=b_0$. Then, for $\rho<1$,
$$|u(\rho,c)-b_0|=|u(\rho,c)-u(1,c)|\leq C\int_{\rho}^1 |\log (1-\rho)|d\rho\leq C(1-\rho)|\log (1-\rho)|.$$
By the equation \eqref{C11'}, for $1/2<\rho<1$,
$$\left|(\rho^{N-1}u'(\rho,c))'\right|\leq C|\log(1-\rho)|.$$
This proves that $u'(\rho,c)$ has a limit $a$ as $\rho\to 1$. By Lemma \ref{L:local_uniq}, $u(\rho,c)=U(\rho,a)$ as announced.
\end{proof}

\subsection{Shooting method}

We denote by $\NNN(c)$ the number of zeros of $w(\cdot,c)=\frac{u(\cdot,c)}{u_{\infty}}-1$ in $(0,1)$.
\begin{lem}
\label{L:N(c)}
Let $c_0\in \RR$ such that $u(1,c_0)>0$ and $u(1,c_0)\neq b_0$. Then $\NNN(c)=\NNN(c_0)$ for $c$ close to $c_0$.
\end{lem}
\begin{proof}
 Since $u(1,c_0)\neq b_0$, we have $w(1,c_0)\neq 0$. By the continuity of $w$, there exists $\rho_{1}\in (1/2,1)$ and $\eps>0$ such that
 $$\rho_{1}\leq \rho\leq 1\text{ and }|c-c_0|\leq \eps\Longrightarrow w(\rho,c)\neq 0.$$
 Thus for $c\in [c_0-\eps,c_0+\eps]$, $\rho \in (\rho_{1},1)$, $\frac{\Theta(\rho,c)}{\pi}-\frac 12$ is not an integer. Since $\rho \mapsto \frac{\Theta(\rho,c)}{\pi}-\frac 12$ is continuous, we deduce that $\ent{\frac{\Theta(\rho,c)}{\pi}-\frac 12}$ is constant for $c\in [c_0-\eps,c_0+\eps]$, $\rho\in (\rho_{1},1)$.

 Next, we can choose $\rho_2>0$ close to $0$ such that $w(\rho,c)\neq 0$ for $\rho\in [0,\rho_2]$, $c\in [c_0-\eps,c_0+\eps]$. By the preceding argument, $\ent{\frac{\Theta(\rho_2,c)}{\pi}-\frac 12}$ is independent of $c$ for $c\in [c_0-\eps,c_0+\eps]$.

 As a conclusion,
 $$\ent{\frac{\Theta(\rho,c)}{\pi}-\frac 12}-\ent{\frac{\Theta(\rho_2,c)}{\pi}-\frac 12},$$
 does not depend on $c\in [c_0-\eps,c_0+\eps]$, $\rho\in (\rho_1,1)$, and the conclusion of the lemma follows from Lemma \ref{L:zeros}.
 \end{proof}
We now proceed with the proof of Theorem \ref{T:critical_case}. We first notice that
\begin{equation}
 \label{zeros_infinity}
 \lim_{c\to\infty}\NNN(c)=+\infty.
\end{equation}
Indeed since
$\Theta(0,c)=\pi$,
we have by Lemma \ref{L:zeros}
$$\NNN(c)\geq -\ent{\frac{\Theta(1/2,c)}{\pi}-\frac{1}{2}}$$
and \eqref{zeros_infinity} follows from Proposition \ref{prop0}.

Let $c_1>0$ such that
$$ c\geq c_1\Longrightarrow \NNN(c)\geq n_{0}.$$
According to Lemma \ref{L:C40}, we are reduced to prove that there exists $c\geq c_1$ such that $u(1,c)=b_0$. We argue by contradiction, assuming that $u(1,c)\neq b_0$ for all $c\geq c_1$. By Corollary \ref{Cor0}, taking a larger $c_1$ if necessary, $u(1,c)$ is close to $b_0$, and thus positive, for $c\geq c_1$. By Lemma \ref{L:N(c)}, $\NNN(c)$ is constant for $c\geq c_1$. This contradicts \eqref{zeros_infinity}, concluding the proof. \qed

\subsection{Limit of the derivative at the boundary of the wave cone}

We conclude this section by giving a property of regular solutions of \eqref{PDE} on $[0,1]$ that will be useful to study the asymptotics of these solutions for large $\rho$.
\begin{prop}
 \label{P:limitU'}
 Assume $N\geq 4$ and $p=1+\frac{4}{N-3}$. For every $n\in \NN$, let $u_n$ be a regular solution of \eqref{PDE} on $[0,\rho_n]$, $\rho_n>1$. Assume
 \begin{equation}
  \label{Cr10} \lim_{n\to\infty} u_n(0)=+\infty,\quad \forall n,\;
  u_n(1)=b_{\infty}.
  \end{equation}
 Then
 $$\lim_{n\to\infty}u_n'(1)=u_{\infty}'(1)=-\alpha b_{\infty}.$$
\end{prop}
\begin{proof}
 Let $v_n=\frac{u_n}{u_{\infty}}$, $a_n=u_n'(1)$. We have
 \begin{align}
  \label{Cr11}
  \left(\rho^2v_n'\right)'&=\frac{(N-3)(N-1)}{4(1-\rho^2)}v_n\left(1-|v_n|^{p-1}\right)\\
 \label{Cr12}
 v_n(1)&=1,\quad v_n'(1)=\alpha+\frac{a_n}{b_{\infty}}.
 \end{align}
By uniqueness in Proposition \ref{P:local2}, $v_n'(1)\neq 0$. We must prove
\begin{equation}
\label{Cr14}
\lim_{n\to\infty} v_n'(1)=0,
\end{equation}
which is equivalent to the conclusion of the proposition. Extracting subsequences, we may assume that the sign of $v_n'(1)$ is independent of $n$. We will assume to fix ideas
\begin{equation}
 \label{vn'1positive}
 \forall n,\quad v_n'(1)>0,
\end{equation}
and will comment on the other case in the end of the proof.

We note that \eqref{vn'1positive} implies that for large $n$, there exists $\sigma_n\in (0,1)$ such that
\begin{equation}
 \label{exist_sigma_n}
v_n'(\sigma_n)=0\quad \text{and}\quad \forall \rho\in (\sigma_n,1),\; v_n'(\rho)>0.
\end{equation}
Indeed, if \eqref{exist_sigma_n} does not hold, $v_n$ is monotonically increasing on $[0,1]$, which implies $v_n(\rho)<1$ for all $\rho\in (0,1)$, contradicting the fact that by Proposition \ref{prop0}, $u_n$ intersects $u_{\infty}$ on $(0,1)$ for large $n$.
We divide the proof of \eqref{Cr14} into two steps.

\medskip

\noindent \emph{Step 1.} We first prove
\begin{equation}
 \label{Cr13}
 \limsup_{n\to\infty}\sigma_n<1.
\end{equation}
We argue by contradiction, assuming that there exists a subsequence of $(\sigma_n)_n$, that we still denote by $(\sigma_n)_n$ such that
\begin{equation}
\label{absurd_sigman}
\lim_{n\to\infty}\sigma_n=1.
\end{equation}
In all the proof, we will denote by $C>0$ a large constant, depending only on $N$, that may change from line to line.
By Corollary \ref{Cor0}, we have that for large $n$,
\begin{equation}
 \label{Cr20}
 \forall \rho\in [\sigma_n,1],\quad \frac 12\leq v_n(\rho)\leq 1.
\end{equation}
By the equation \eqref{Cr11}, we have that $(\rho^2 v_n')'>0$ on $(\sigma_n,1)$ which yields
\begin{equation}
 \label{Cr21}
 \forall \rho\in (\sigma_n,1),\quad 0<\rho^2v'_n(\rho)<v_n'(1).
\end{equation}
Next, we see by \eqref{absurd_sigman}, \eqref{Cr20} and \eqref{Cr21} that for large $n$, for all $\rho\in [\sigma_n,1]$, we have,
$$ |1-v_n^{p-1}(\rho)|=(p-1)\left| \int_{\rho}^1 v_n'(\sigma)v_n^{p-2}(\sigma)d\sigma\right|\leq C(1-\rho)v_n'(1).$$
Going back to \eqref{Cr11}, we see that for large $n$;
$$v_n'(1)=|v_n'(1)-\sigma_n^2v_n'(\sigma_n)|\leq C\int_{\sigma_n}^1 v_n'(1)d\sigma\leq C(1-\sigma_n)v_n'(1).$$
Hence $\sigma_n\leq 1-1/C$ for large $n$, which contradicts \eqref{absurd_sigman} and yields \eqref{Cr13}.

\medskip

\noindent\emph{Step 2: conclusion of the proof.} By the preceding step, there exists a small $\eps>0$ such that
$$ \forall n,\quad \sigma_n\leq 1-\eps.$$
Using Corollary \ref{Cor0} and the same argument as before,
\begin{equation}
 \label{Cr30}
 \forall \rho\in [1-\eps,1], \quad \frac 12\leq v_n(\rho)\leq 1\text{ and } 0<\rho^2v_n'(\rho)<v_n'(1).
\end{equation}
This yields $|1-v_n^{p-1}(\rho)|\leq C(1-\rho)v_n'(1)$ for $\rho\in [1-\eps,1]$ and thus, integrating \eqref{Cr11},
\begin{equation}
 \label{Cr31}
 \left|v_n'(1)-(1-\eps)^{2}v_n'(1-\eps)\right|\leq C\eps v_n'(1).
\end{equation}
By Corollary \ref{Cor0},
$$\lim_{n\to\infty} v_n'(1-\eps)=0.$$
Assuming that $\eps>0$ is so small that $C\eps<1$ in the right-hand side of \eqref{Cr31}, we obtain \eqref{Cr14}, concluding the proof when $v_n'(1)>0$ for all $n$.

The proof is almost the same in the case where $v_n'(1)<0$ for all $n$. The starting point is to prove that there exists $\sigma_n\in (0,1)$ such that
$$ v_n'(\sigma_n)=0\text{ and }\forall \rho\in (\sigma_n,1),\; v_n'(\rho)<0,$$
which again follows from the fact that $u_n$ intersects $u_{\infty}$ on $(0,1)$ due to $\lim_{\rho\rightarrow0+}v_{n}(\rho)=0$.
\end{proof}

\section{Extension beyond the past light cone}
\label{S:extension}
In this Section, we will show that the solutions constructed in Theorems \ref{Thm0} and \ref{T:critical_case} can be extended for large $n$ beyond the past light cone to infinity, i.e., Theorem \ref{Thm1}. This is a consequence of the following proposition that describes the behaviour of $U(\rho,b)$ for $\rho>1$:
\begin{prop}\label{P:A0}
 Assume $N\geq 3$ and $1+\frac{4}{N-2}<p<1+\frac{4}{N-3}$. There exists $b_*\in [0,b_{\infty})$ with the following property. Let, for $b>0$, $U(\rho)=U(\rho,b)$ be the $C^2$ solution of \eqref{PDE} with $U(1)=b$ defined by Proposition \ref{P:local1}, and $[1,\rho_+)$ its forward maximal interval of existence. Then:
 \begin{itemize}
  \item If $b>b_0$, then $\rho_+<+\infty$ and for all $\rho\in (1,\rho_+)$, $U(\rho)>b_0$.
  \item If $b_*< b<b_0$, then $\rho_+=+\infty$,
  $$\forall \rho>1, \quad 0<U(\rho)<b_0,\;U'(\rho)<0,$$
  and there exists $L>0$ such that
  $$\lim_{\rho\to\infty} \rho^{\alpha}U(\rho)=L,\quad \lim_{\rho\to\infty}\rho^{\alpha+1}U'(\rho)=-\alpha L.$$
 \end{itemize}
 If $b>b_{\infty}$, then $L>b$ and if $b_*<b<b_{\infty}$ then $L<b$.
Finally, if $N=3$, we can take $b_*=0$.
\end{prop}
\begin{rem}
 Proposition \ref{P:A0} classifies all the regular solutions of \eqref{PDE} for $\rho\geq 1$ when $N=3$, completing the work of Kavian and Weissler \cite{KW} (see Theorem \ref{T:A50} below). When $N\geq 4$, we do not know if the best possible value of $b_*$ is $0$ (which would also yield a complete classification of regular solutions of \eqref{PDE}) or strictly positive.
\end{rem}
\begin{rem}
\label{R:BMW}
 Proposition \ref{P:A0} in dimension $N=3$ is contained in \cite{BMW}. In this article, it is also claimed that  any solution of \eqref{PDE} that is regular on $[0,1]$ satisfies $|u(1)|<b_0$ (see \cite[Proposition 1]{BMW}). This would imply, together with the case $N=3$ in Proposition \ref{P:A0}, that all regular solutions of \eqref{PDE} are global when $N=3$. However there is a sign mistake in the proof of \cite[Proposition 1]{BMW}: in the inequality (34) of this proof, the authors bound $b_0-u(s)$ for $s<\rho$ by $b_0-u(\rho)$ in a region where $s\mapsto u(s)$ is monotonically increasing. We were not able to fill this gap, which explains why we have restricted the global existence result in dimension $N=3$ to the solutions $u_n$ with $n$ odd (see Remark \ref{R:N=3}), giving a slightly weaker statement than in \cite[Section 4]{BMW}. Let us mention however that numerical investigations \cite{BizonPC} suggest that 
that $|u(1)|<b_0$ holds for any regular solutions of \eqref{PDE}, and thus that all the solutions of Theorem \ref{Thm0} are global in dimension $N=3$.\ednote{{\color{blue}Thomas: This remark was slightly reformulated. I cited numerical investigations from the personal communication of Bizo\'n}}
 \end{rem}

\begin{prop}
\label{P:A1}
Assume $N\geq 4$ and $p=1+\frac{4}{N-3}$. There exists $a_*\in [-\infty,-\alpha b_0)$ with the following property. Let for $a\in \RR$ the solution of \eqref{PDE} such that $U(1)=b_0$, $U'(1)=a$ given by Proposition \ref{P:local2}, and $[1,\rho_+)$ its forward maximal interval of existence. Then:
 \begin{itemize}
  \item If $a>0$, then $\rho_+<+\infty$ and for all $\rho\in (1,\rho_+)$, $U(\rho)>b_0$.
  \item If $a_*<a<0$, then $\rho_+=+\infty$,
  $$\forall \rho>1, \quad 0<U(\rho)<b_0,\;U'(\rho)<0$$
  and there exists $L>0$ such that
  $$\lim_{\rho\to\infty} \rho^{\alpha}U(\rho)=L,\quad \lim_{\rho\to\infty} \rho^{\alpha+1}U'(\rho)=-\alpha L.$$
 \end{itemize}
\end{prop}
\begin{rem}
 In the case $N=5$, the solution \eqref{SolGlSc} of Glogi\'c and Sch\"orkhuber shows that $a_*>-\infty$.
\end{rem}
\subsection{Blow-up in finite time}
We prove Propositions \ref{P:A0} and \ref{P:A1} together. The proofs are divided into a few lemmas. In this subsection, we consider the cases $b>b_0$ and $a>0$. The other cases are treated in the next subsection.
\begin{lem}
 Assume $1+\frac{4}{N-2}<p<1+\frac{4}{N-3}$ and $b>b_0$. Then the solution $U=U(\cdot,b)$ of \eqref{PDE} blows up in finite time, i.e. $\rho_+<\infty$. The same conclusion holds, if $N\geq 4$, $p=1+\frac{4}{N-3}$ and $a>0$ for the solution $U=U(\cdot,a)$.
\end{lem}
\begin{proof}
 As before, we write the equation \eqref{PDE} in self-adjoint form
 \begin{equation}
  \label{A12}
  \left( \rho^{N-1} (\rho^2-1)^{{\alpha}-\frac{N-3}{2}}U' \right)'=-\rho^{N-1}(\rho^2-1)^{\alpha-\frac{N-1}{2}} \left(b_0^{p-1}-|U|^{p-1}\right)U.
 \end{equation}
 Using \eqref{A12} and a standard bootstrap argument we see that for any $\sigma\geq 1$,
\begin{equation}
 \label{A13}
 \Big( U(\sigma) \geq b_0\text{ and } U'(\sigma)> 0 \Big)\Longrightarrow \forall \rho\in [\sigma,\rho_+), \quad U(\rho)>b_0\text{ and } U'(\rho)>0.
\end{equation}
In the case $p<1+\frac{4}{N-3}$, recall that (see Proposition \ref{P:local1})
\begin{equation}
 \label{A14}
 U'(1)=\frac{b(b_0^{p-1}-b^{p-1})}{N-3-2\alpha}.
\end{equation}
By the assumption $p<1+\frac{4}{N-3}$, we have $N-3-2\alpha<0$, and thus $U'(1)>0$. As a consequence of \eqref{A13}
\begin{equation}
 \label{A20}
 \forall \rho\in (1,\rho_+),\quad U(\rho)>b_0,\; U'(\rho)>0.
\end{equation}
If $p=1+\frac{4}{N-3}$ we have $U(1)=b_0$ and by the assumption of the lemma, $U'(1)>0$ and we also obtain \eqref{A20} as a consequence of \eqref{A13}.

We will prove that $\rho_+<\infty$ by contradiction. Assume $\rho_+=+\infty$ and consider as in \eqref{2-3-0}.
$$H(\rho)=-(\rho^2-1)\frac{(U')^2}{2}+\frac{|U|^{p+1}}{p+1}-\frac{p+1}{(p-1)^2}U^2.$$
By the explicit value of $H'(\rho)$ given in \eqref{2-3-1}, we see that
\begin{equation}
 \label{A23}
 \rho>\sqrt{(N-1)\frac{p-1}{p+3}}\Longrightarrow H'(\rho)\geq 0.
\end{equation}
As a consequence, $H$ has a limit $H_{\infty}\in (-\infty,+\infty]$ as $\rho\to\infty$.

By \eqref{A20}, $U$ has a limit $U_{\infty}\in (b_0,+\infty]$ as $\rho\to\infty$. If $H_{\infty}<\infty$, then $\rho^2(U'(\rho))^2$ has a limit $\ell\in [0,+\infty]$ as $\rho\to+\infty$.
Integrating the expression of $H'$ in \eqref{2-3-1}, we obtain $\int_1^{+\infty} \rho (U'(\rho))^2<\infty$, and thus that $\ell=0$. Going back to the equation \eqref{PDE}, we see, since $U_{\infty}>b_0$,  that $U''(\rho)>0$ for large $\rho$. Thus $U'$ is monotonically increasing for large $\rho$, a contradiction with the fact that $\rho U'(\rho)$ converges to $0$.  Hence
\begin{equation}
 \label{A30}
 \lim_{\rho\to\infty} H(\rho)=+\infty,
\end{equation}
and, by the definition of $H(\rho)$,
\begin{equation}
 \label{A31}
 \lim_{\rho\to\infty} U(\rho)=+\infty.
\end{equation}
We next use a change of variables from \cite{KW}. Let $\rho=e^{s}$, $z(s)=U(e^s)$. Then
\begin{equation}
 \label{A32}
 z''+(2\alpha+1)z'-\left( [z|^{p-1}-b_0^{p-1} \right)z=e^{-2s}\left( z''+(N-2)z' \right).
\end{equation}
By \eqref{A30}, for large $\rho$,
\begin{equation}
 \label{A40} U(\rho)^{\frac{p+1}{2}}\geq \frac{\rho}{2}U'(\rho).
\end{equation}
Hence for large $s$,
\begin{equation}
 \label{A41}
 z(s)^{\frac{p+1}{2}}\geq \frac{1}{2}z'(s).
\end{equation}
Since $p>\frac{p+1}{2}$, we deduce from the equation \eqref{A32} and the inequality \eqref{A41}
\begin{equation}
 \label{A42}
 z''(s)\geq \frac{1}{2} z^p(s).
\end{equation}
Multiplying \eqref{A42} by $z'(s)$, we obtain
$$ \frac{d}{ds}\left( (z')^2-\frac{1}{p+1}z^{p+1} \right)\geq 0.$$

Hence, using that $\lim_{s\to\infty}z(s)=+\infty$, we deduce that, for large $s$,
$$z'(s)\geq \frac{1}{\sqrt{p+2}}{z^{\frac{p+1}{2}}}.$$
Thus
$$ \frac{d}{ds}\frac{1}{z^{\frac{p-1}{2}}(s)}\leq -\frac{p-1}{2\sqrt{p+2}}$$
for large $s$, a contradiction since $z^{\frac{p-1}{2}}$ is positive.
\end{proof}

\subsection{Global existence}
In this subsection we treat the cases $0<b<b_0$ and $a<0$. We recall the following result from the work of Kavian and Weissler (see \cite[Theorem 3.1]{KW}).
\begin{thm}
 \label{T:A50}
 Assume $1+\frac{4}{N-2}<p$.
 Let $U$ be a nonzero solution of \eqref{PDE} on $(1,+\infty)$ with  $\rho_+=+\infty$ and
 \begin{equation}
  \label{A50}
  \limsup_{\rho \to +\infty} |U(\rho)|<b_0.
 \end{equation}
 Then there exists $L\in \RR \setminus \{0\}$ such that one of the following holds
 \begin{equation}
  \label{A51}
  \lim_{\rho\to+\infty} \rho^{\alpha} U(\rho)=L,\quad \lim_{\rho \to+\infty} \rho^{\alpha+1}U'(\rho)=-\alpha L
  \end{equation}
or
\begin{equation}
 \label{A52}
 \lim_{\rho\to+\infty} \rho^{\alpha+1} U(\rho)=L,\quad \lim_{\rho\to+\infty} \rho^{\alpha+2}U'(\rho)=-(\alpha+1)L.
\end{equation}
\end{thm}
To understand the heuristic of Theorem 3.4, let as before $z(s)=U(e^s)$. Theorem 3.4 says exactly that assuming \eqref{A50}, the solution $z$ is close, for large $s$, to one of the solutions of the linear autonomous equation obtained by neglecting the nonlinear term and the right-hand side of equation \eqref{A32}.

We next show:
\begin{claim}
\label{Cl:A50}
Assume that one of the following holds:
\begin{itemize}
  \item $1+\frac{4}{N-2}<p<1+\frac{4}{N-3}$, $0<b<b_0$ or
  \item $N\geq 4$, $p=1+\frac{4}{N-3}$, $a<0$,
\end{itemize}
Let $U(\rho)=U(\rho,b)$ in the first case and $U(\rho)=U(\rho,a)$ in the second case. Assume that there exists $\rho_1\in (1,\rho_+)$ such that $U(\rho)> 0$ for $1\leq \rho\leq \rho_1$. Then
 \begin{equation}
  \label{A53}
  \forall \rho\in (1,\rho_1], \quad U(\rho)<b_0,\; U'(\rho)<0.
 \end{equation}
\end{claim}
\begin{proof}
We note that
 \begin{equation}
  \label{A54}
  U'(1)<0.
 \end{equation}
 This is exactly the assumption $a<0$ in the case $p=1+\frac{4}{N-3}$. In the case $p<1+\frac{4}{N-3}$, it follows from the formula \eqref{A14} for $U'(1)$ and the assumption $0<b<b_0$.

 By the equation in self-adjoint form \eqref{A12},
 \begin{equation}
  \label{A60}
  0<U(\rho)<b_0\Longrightarrow \left( \rho^{N-1}( \rho^2-1)^{\alpha-\frac{N-3}{2}}U' \right)'<0.
 \end{equation}
 The conclusion \eqref{A53} of the claim follows from $0<U(1)\leq b_0$,  \eqref{A54}, \eqref{A60} and a standard bootstrap argument.
\end{proof}
The case $N=3$ is treated in \cite{BMW} (using also the work of \cite{KW} to obtain the exact asymptotics of $U$). We recall their argument for the sake of completeness.
\begin{lem}
\label{L:A60}
Assume $N=3$, $p>5$ and $0<b<b_0$. Let $U(\rho)=U(\rho,b)$. Then
\begin{gather}
 \label{A61} \forall \rho>1,\quad 0<U(\rho)<b_0\\
 \label{A62} \exists L\in(0,\infty),\quad \lim_{\rho\to\infty}\rho^{\alpha}U(\rho)=L,\quad
\end{gather}
\end{lem}
\begin{proof}
First, we will show that $0<U(\rho)<b_{0}$ for any $\rho>1$. To this end, let us define
\begin{equation}\label{4-0}
  \bar{\rho}:=\sup\Big\{\rho\in(1,+\infty)\;\big|\; \forall s\in[1,\rho),\; 0<U(s)\Big\}>1.
\end{equation}
We aim to prove that $\bar{\rho}=+\infty$ by contradiction arguments.

By Claim \ref{Cl:A50},
\begin{equation}
 \label{rhobar}
 \forall \rho\in (1,\overline{\rho}), \quad 0<U(\rho)<b_0\text{ and }U'(\rho)<0.
\end{equation}
Now suppose on the contrary that $\bar{\rho}<+\infty$, then one has $U(\bar{\rho})=0$. Moreover, by \eqref{rhobar} and the fact that $U$ is not identically $0$,
\begin{equation}\label{4-3}
  U^{\prime}(\bar{\rho})<0.
\end{equation}
Let us consider the function $h(\rho):=2\rho U'(\rho)+(\alpha+1)U(\rho)$ for $\rho>1$. Noting that $h(1)=\frac{b^{p}}{\alpha}>0$, we can prove that $h(\rho)>0$ for any $\rho\in[1,\bar{\rho})$. Suppose not. Thus there exists a $\tilde{\rho}\in(1,\bar{\rho})$ such that $h(\tilde{\rho})=0$ and $h(\rho)>0$ for any $\rho\in[1,\tilde{\rho})$, and equation \eqref{PDE} yields
\begin{equation}\label{4-2}
h'(\tilde{\rho})=\frac{1-\alpha^{2}}{2\tilde{\rho}}U(\tilde{\rho})+\frac{2\tilde{\rho}}{\tilde{\rho}^{2}-1}U^{p}(\tilde{\rho})>0,
\end{equation}
which is absurd. Consequently, we have
\begin{equation}\label{4-4}
  U(\bar{\rho})=\frac{h(\bar{\rho})-2\bar{\rho}U'(\bar{\rho})}{\alpha+1}\geq-\frac{2\bar{\rho}}{\alpha+1}U'(\bar{\rho})>0,
\end{equation}
which contradicts with $U(\bar{\rho})=0$. Thus we must have $\bar{\rho}=+\infty$.

We are thus in the setting of Theorem \ref{T:A50}. Furthermore, since $h(\rho)>0$ for all $\rho>1$, we have
$$\forall \rho>1, \quad \frac{U'(\rho)}{U(\rho)}>-\frac{\alpha+1}{2\rho}>-(\alpha+1).$$
Thus we cannot be in case \eqref{A52} of Theorem \ref{T:A50}, and \eqref{A51} must hold.
The conclusion of Lemma \ref{L:A60} follows.
\end{proof}
Note that Lemma \ref{L:A60} yields the conclusion of Proposition \ref{P:A0} for $N=3$, except for the fact that $L>b$ if $b>b_{\infty}$ and $L<b$ if $b<b_{\infty}$. This last fact will follow from the arguments below (see Lemma \ref{L:70} and \ref{L:A90}).

In the general case, the proof of Lemma \ref{L:A60} falls down, and we will use more intricate arguments. The proof is easier assuming that $U(\rho)$ is above $u_{\infty}(\rho)$ for $\rho>1$ close to $1$, which is exactly the meaning of the assumptions of the next lemma:
\begin{lem}
 \label{L:70}
 Assume $N\geq 3$, $1+\frac{4}{N-2}<p<1+\frac{4}{N-3}$ and $b_{\infty}<b<b_0$, or $N\geq 4$, $p=1+\frac{4}{N-3}$ and $-\alpha b_0<a<0$. Then $\rho_{+}=+\infty$,
 $$\forall \rho>1,\quad \frac{b_{\infty}}{\rho^{\alpha}}<U(\rho)<b_0,$$
 and there exists $L>b$ such that
 $$\lim_{\rho\to+\infty} \rho^{\alpha}U(\rho)=L,\quad \lim_{\rho\to+\infty} \rho^{\alpha+1}U'(\rho)=-\alpha L.$$
\end{lem}
\begin{proof}
 Consider $V=\frac{U}{u_{\infty}}=\rho^{\alpha}\frac{U}{b_{\infty}}$. Using that $U$ satisfies \eqref{PDE}, we obtain that $V$ satisfies the following equation
 \begin{multline}
  \label{A70}
  \left( \rho^{N-1-2\alpha}\left(\rho^2-1\right)^{\alpha-\frac{N-3}{2}}V' \right)'\\
  =-\alpha(N-2-\alpha)\rho^{N-2\alpha-3} \left( \rho^2 -1\right)^{\alpha-\frac{N-1}{2}} V\left( 1-|V|^{p-1} \right).
 \end{multline}
 First assume $p<1+\frac{4}{N-3}$.
 By our assumption
 \begin{equation}
  \label{A71}
  V(1)>1.
 \end{equation}
 Furthermore,
 \begin{equation}
  \label{A72}
  V'(1)=\frac{1}{b_{\infty}} \left( \alpha b+U'(1) \right)=\frac{b}{b_{\infty}}\left( \alpha+\frac{b_0^{p-1}-b^{p-1}}{N-3-2\alpha} \right)=\frac{b}{b_{\infty}}\left( \frac{b_{\infty}^{p-1}-b^{p-1}}{N-3-2\alpha} \right).
 \end{equation}
 Using that $N-3-2\alpha<0$, and that $b>b_{\infty}$ we deduce
 \begin{equation}
  \label{A73}
  V'(1)>0.
 \end{equation}
 If $N\geq 4$ and $p=1+\frac{4}{N-3}$ we have
 \begin{equation}
 \label{A73'}
 V(1)=1\text{ and } V'(1)=\frac{1}{b_{0}}(\alpha b_0+a)>0.
 \end{equation}
 By \eqref{A70},
 $$V(\rho)>1\Longrightarrow \left( \rho^{N-1-2\alpha}\left(\rho^2-1\right)^{\alpha-\frac{N-3}{2}}V' \right)'>0.$$
 By a straightforward bootstrap argument
 \begin{equation}
  \label{A80}
  \forall \rho\in (1,\rho_+),\quad V'(\rho)>0 \text{ and } V(\rho)>1.
 \end{equation}
 This proves that $U(\rho)>\frac{b_{\infty}}{\rho^{\alpha}}>0$ for all $\rho\in (1,\rho_+)$. Combining with Claim \ref{Cl:A50}, we deduce $\rho_+=+\infty$ and
 \begin{equation}
  \label{A81}
  \forall \rho\in (1,+\infty), \quad \frac{b_{\infty}}{\rho^{\alpha}}<U(\rho)<b_0, \quad U'(\rho)<0.
 \end{equation}
 Thus we are in case \eqref{A51} of Theorem \ref{T:A50}, and the conclusion follows. Note that in the case $p<1+\frac{4}{N-3}$
 $$L=b_{\infty}\lim_{\rho\to\infty}V(\rho)>b,$$
 since $V(1)=b/b_{\infty}$ and $V$ increases, and the same argument shows that $L>b_{\infty}$ in the case $p=1+\frac{4}{N-3}$.
\end{proof}
It remains to treat the case $0<b<b_{\infty}$ and $a<-\alpha b_{\infty}$ which are the most difficult ones since there does not seem to exist a simple general argument in these cases to prove that $U$ is positive. We will use ideas from \cite{KW}. We define $\beta_{N,p}$, $\rho_{N,p}$ and $\widetilde{B}(\rho)$, for $\rho>1$, by
\begin{gather}
\label{deftB}
\widetilde{B}(\rho):=\frac{\frac{p+3}{p-1}-(N-2)\rho^{-2}}{1-\rho^{-2}},\quad \beta_{N,p}:=\frac{(2N-8)p^2+(24-12N)p+10N}{(p-1)^2}\\
\label{defrho}
\frac{1}{\rho_{N,p}^2}:=\frac{-\beta_{N,p}-\sqrt{\beta_{N,p}^2-4(N-2)^2}}{2(N-2)^2}.
\end{gather}
We will prove below that $\beta_{N,p}\leq -2(N-2)$, so that $\rho_{N,p}$ is well defined.

\begin{lem}
 \label{L:A90}
 Assume $N=3$ or $N\geq 4$ and $1+\frac{4}{N-2}<p\leq 1+\frac{4}{N-3}$.
 Let $\rho_0>1$, $U$ be a solution of \eqref{PDE} in a neighborhood of $\rho_0$, and $[\rho_0,\rho_+)$ its maximal forward interval of existence. Assume
 \begin{gather}
 \label{A90}
  0<U(\rho_0)<b_{\infty},\quad U'(\rho_0)<0\\
  \label{A91}
  \rho_0\frac{U'(\rho_0)}{U(\rho_0)}>-\frac{\widetilde{B}(\rho_0)}{2}\\
  \label{A92}
  \rho_0>\rho_{N,p}
\end{gather}
Then $\rho_+=+\infty$ and
$$\forall \rho>\rho_{0},\quad U(\rho)>0\text{ and }\frac{\rho U'(\rho)}{U(\rho)}>-\frac{\widetilde{B}(\rho)}{2}.$$
\end{lem}
\begin{proof}
 The proof is a refinement of the proof of \cite[Proposition 3.5]{KW}.

 \noindent\emph{Step 1}. Let as before $z(s)=U(e^s)$, and recall that $z$ is solution of the equation \eqref{A32}. We rewrite this equation as follows
 \begin{equation}
   \label{A100}
   z''+B(s)z'+C(s)z=0,
 \end{equation}
 where
 \begin{equation}
   \label{A101}
   B(s):=\frac{\frac{p+3}{p-1}-(N-2)e^{-2s}}{1-e^{-2s}},\quad C(s):=\frac{b_0^{p-1}-|z(s)|^{p-1}}{1-e^{-2s}}.
 \end{equation}
Note that $B(s)=\widetilde{B}(e^s)$. When $z(s)\neq 0$, we define
$$r(s)=\frac{z'(s)}{z(s)}.$$
The equation \eqref{A100} yields
\begin{equation}
 \label{A102}
 r'(s)=-\left(r^2(s)+B(s)r+C(s)\right).
\end{equation}
By assumption \eqref{A91},
\begin{equation}
 \label{A110}
 r(s_0)>-\frac{B(s_0)}{2},
\end{equation}
where $s_0=\log\rho_0$. We claim that if $s_0$ is larger than $s_{N,p}=\log \rho_{N,p}$, we have
\begin{equation}
 \label{A111}
 \forall s\geq s_0,\quad z(s)\neq 0\text{ and } r(s)>-\frac{B(s)}{2}.
\end{equation}
Indeed, assume that \eqref{A111} does not hold, and let $s_1>s_0$ be the first number such that $r(s_1)=-\frac{B(s_1)}{2}$ or $z(s_1)=0$. Thus
\begin{equation}
 \label{A112}
 \forall s\in [s_0,s_1), \quad r(s)>-\frac{B(s)}{2}\text{ and }z(s)\neq 0.
\end{equation}
If $z(s_1)=0$, then $z'(s_1)\neq 0$ (since $z$ is not identically $0$) and $z'(s)$ must be negative for $s<s_1$ close to $s_1$, thus $z'(s_1)<0$, which yields
$$\lim_{s\overset{<}{\to} s_1}r(s)=-\infty,$$
contradicting \eqref{A112}. Thus $z(s_1)>0$ and by the definition of $s_1$, we must have $r(s_1)=-\frac{B(s_1)}{2}$. In view of \eqref{A112}, this implies $r'(s_1)\leq -\frac{B'(s_1)}{2}$, that is
\begin{equation}
 \label{A120}
 \frac{B(s_1)^2}{4}-C(s_1)+\frac{B'(s_1)}{2}\leq 0.
\end{equation}
By the definitions \eqref{A101} of $B(s)$ and $C(s)$,
\begin{equation*}
 \frac{B(s_1)^2}{4}+\frac{B'(s_1)}{2}-C(s_1)=\frac{\TTT\left( e^{-2s_1} \right)}{4(1-e^{-2s_1})^2}+\frac{|z(s)|^{p-1}}{1-e^{-2s_1}},
\end{equation*}
where
$$\TTT(X)=(N-2)^2X^2+\frac{(2N-8)p^2+(24-12N)p+10N}{(p-1)^2}X+1.$$
We claim that the assumption $1+\frac{4}{N-2}<p<1+\frac{4}{N-3}$ implies that $\mathcal{T}$ has two positive roots. Indeed we have
$$\mathcal{T}(0)=1,\quad \mathcal{\TTT}(1)=1+(N-2)^2+\frac{(2N-8)p^2+(24-12N)p+10N}{(p-1)^2}.$$
Note that, since $p>1+\frac{4}{N-2}$,
$$\frac{d}{dp}\TTT(1)=\frac{8}{(p-1)^3}\left( p(N-1)-N-3 \right)>0.$$
Thus we can bound $\TTT(1)$ from above by its value for $p=1+\frac{4}{N-3}$. An explicit computation yields that $\TTT(1)=0$ in this case, which implies that for $1+\frac{4}{N-2}<p<1+\frac{4}{N-3}$, $\TTT(1)<0$.
Since the signs of $\TTT(0)$ and $\TTT(1)$ are opposite, $\TTT$ has at least one root lying in the interval $(0,1)$. Furthermore the product of the two roots (that may coincide) is equal to $\frac{\TTT(0)}{(N-2)^{2}}=\frac{1}{(N-2)^{2}}$, and thus $\beta_{N,p}\leq-2(N-2)$ so $\rho_{N,p}$ is well defined, where $\beta_{N,p}$ and $\rho_{N,p}$ are defined in \eqref{deftB} and \eqref{defrho} respectively. More explicitly, the root between $0$ and $\frac{1}{N-2}$ is exactly $\rho^{-2}_{N,p}$.

As a conclusion, since by our assumption $0<\rho_0^{-2}=e^{-2s_0}<\frac{1}{\rho^2_{N,p}}$, the inequality $s_1>s_0$, implies $\mathcal{T}(e^{-2s_1})>0$, and thus \eqref{A120} cannot hold, which proves that \eqref{A111} holds for all $s>s_{0}$. Going back to the function $U$, we see that for $\rho>\rho_0$,
$$U(\rho)>0,\quad \rho \frac{U'(\rho)}{U(\rho)}>-\frac{\widetilde{B}(\rho)}{2}.$$
\end{proof}
\begin{lem}
 \label{L:A150}
 Assume $N=3$ or $1+\frac{4}{N-2}<p<1+\frac{4}{N-3}$.
 There exists $b_*\in [0,b_{\infty})$ such that if $b_*<b<b_{\infty}$, then $U=U(\cdot,b)$ is defined and positive on $[1,+\infty)$ and there exists $L\in (0,b)$ such that
 \begin{equation}
 \label{A150defL}
 \lim_{\rho\to+\infty} \rho^{\alpha}U(\rho)=L,\quad \lim_{\rho\to+\infty}\rho^{\alpha+1}U'(\rho)=-L\alpha.
 \end{equation}
 If $N\geq 4$ and $p=1+\frac{4}{N-3}$, there exists $a_*<-\alpha b_0$ such that, if $a_*<a<-\alpha b_0$ then $U=U(\cdot,a)$ is defined and positive on $[1,+\infty)$, and satisfies \eqref{A150defL} for some $L\in (0,b_0)$.
\end{lem}
\begin{proof}
We consider to fix ideas the case where $1+\frac{4}{N-2}<p<1+\frac{4}{N-3}$. The proof in the case where $p=1+\frac{4}{N-3}$ is almost the same and is omitted.
Note that
$$\lim_{\rho\to+\infty}\widetilde{B}(\rho)=\frac{p+3}{p-1}>2\alpha.$$
We fix $\rho_0>\rho_{N,p}$ such that $\widetilde{B}(\rho_0)>2\alpha$. Letting as usual $u_{\infty}(\rho)=b_{\infty}\rho^{-\alpha}=U(\rho,b_{\infty})$, we see that
\begin{equation}
 \label{A150}
 \rho_0\frac{u_{\infty}'(\rho_0)}{u_{\infty}(\rho_0)}=-\alpha>-\frac{\tilde{B}(\rho_0)}{2}.
\end{equation}
Using the continuity of $(\rho,b)\mapsto U(\rho,b)$ (see Proposition \ref{P:local1}), we obtain that there exists $b_*\in [0,b_{\infty})$ such that
\begin{equation}
 \label{A151}
 b_*<b<b_{\infty}\Longrightarrow \rho_0\frac{U'(\rho_0)}{U(\rho_0)}>-\frac{\widetilde{B}(\rho_0)}{2} \text{ and }\forall \rho\in[1,\rho_{0}],\; 0<U(\rho)<b_{\infty},\; U'(\rho)<0.
\end{equation}
We can thus use Lemma \ref{L:A90}, which implies that $\rho_+=+\infty$ and
\begin{equation}
 \label{A152}
 \forall \rho>1, \; U(\rho)>0 \text{ and }\forall \rho>\rho_0,\; \rho\frac{U'(\rho)}{U(\rho)}>-\frac{\tilde{B}(\rho)}{2}.
\end{equation}
Using the equation \eqref{A70} for $V=U/u_{\infty}$, and that $V'(1)<0$ by \eqref{A72} and the assumption $0<b<b_{\infty}$, we obtain
\begin{equation}
 \label{A153}
 \forall \rho>1,\quad V'(\rho)<0,\; 0<V(\rho)<1.
\end{equation}
In particular,
\begin{equation}
 \label{A154}
 \forall \rho,\quad 0<U(\rho)<\frac{b}{\rho^{\alpha}}.
\end{equation}
Thus $U$ satisfies the assumptions of Theorem \ref{T:A50}. By the second inequality of \eqref{A152} and since $\lim_{\rho\to\infty}\widetilde{B}(\rho)=\frac{p+3}{p-1}$, we have
$$\limsup_{\rho\to+\infty}\frac{\rho U'(\rho)}{U(\rho)}>-\frac{p+3}{2(p-1)}>-\frac{p+1}{p-1}=-(\alpha+1).$$
Thus we must be in case \eqref{A51} of Theorem \ref{T:A50}, which shows that there exists $L\in (0,\infty)$ such that \eqref{A150defL} holds. Using that $V'(\rho)<0$, we see that for all $\rho\geq 2$, $V(\rho)<V(2)<V(1)=b/b_{\infty}$, which implies that $L<b$, concluding the proof.
\end{proof}

\begin{proof}[Proof of Theorem \ref{Thm1}]
Theorem \ref{Thm1} is a immediate consequence of Propositions \ref{P:A0} and \ref{P:A1}.

Let $u\in C^{2}([0,\sigma))$ ($\sigma >1$) be a solution of \eqref{PDE} on $(0,\sigma)$, and assume that $u(0)\geq C(N,p)$ for a large constant $C(N,p)$ to be specified.

First assume $N=3$ or $N\geq 4$ and $1+\frac{4}{N-2}<p<1+\frac{4}{N-3}$. By Corollary \ref{Cor0}, $u(1)$ is close to $b_{\infty}$ if $u(0)$ is large. Taking $C(N,p)$ large enough, we obtain
$$ b_{*} <u(1) < b_0,$$
where $b_*$ is as in Proposition \ref{P:A0} and Lemma \ref{L:A150}, and the conclusion of the theorem follows from Proposition \ref{P:A0}. Note that if $n$ is large, then $u_n(0)\geq C(N,p)$, proving that the theorem applies to $u_n$. This proves Theorem \ref{Thm1} in this case.

If $N=3$ and $n$ is odd (possibly small) then $u_n$ has an even number of intersections with $u_{\infty}$, and thus $0<u_{n}(1)<b_{\infty}$. Thus it also satisfies the conclusion of Proposition \ref{P:A0} (see the last sentence of this proposition), and Remark \ref{R:N=3} follows.

Next, we consider the case $N\geq 4$ and $p=1+\frac{4}{N-3}$. By Proposition \ref{P:limitU'}, $u'(1)$ is close to $-\alpha b_{\infty}$. Taking $C(N,p)$ large enough, we obtain
$$a_*<u'(1)<0,$$
where $a_*$ is as in Proposition \ref{P:A1}. Thus Proposition \ref{P:A1} implies the conclusion of Theorem \ref{Thm1} in this case.
\end{proof}


\begin{thebibliography}{99}

\bibitem{BizonPC} P. Bizo\'{n}. Personal communication.

\bibitem{BizonCMP} P. Bizo\'{n}, {\it Equivariant Self-Similar Wave Maps from Minkowski Spacetime into 3-Sphere}, Comm. Math. Phys., \textbf{215} (2000), no. 1, 45-56.


\bibitem{BBM} P. Bizo\'{n}, P. Breitenlohner and D. Maison, {\it Self-similar solutions of the cubic wave equation}, Nonlinearity, \textbf{23} (2010), no. 2, 225-236.

\bibitem{BCT} P. Bizo\'{n}, T. Chmaj and Z. Tabor, {\it On blow-up for semilinear wave equations with a focusing nonlinearity}, Nonlinearity, \textbf{17} (2004), no. 6, 2187-2202.

\bibitem{BMW} P. Bizo\'{n}, D. Maison and A. Wasserman, {\it Self-similar solutions of semilinear wave equations with a focusing nonlinearity}, Nonlinearity, \textbf{20} (2007), no. 9, 2061-2074.


\bibitem{BFM} P. Breitenlohner, P. Forg\'{a}cs and D. Maison, {\it Static spherically symmetric solutions of the Einstein-Yang-Mills equations}, Comm. Math. Phys., \textbf{163} (1994), no. 1, 141-172.


\bibitem{Donninger17} R. Donninger, {\it Strichartz estimates in similarity coordinates and stable blowup for the critical wave equation}, Duke Math. J., \textbf{166} (2017), no. 9, 1627-1683. \ednote{{\color{red} Wei: I have corrected a few typos in the references 5-8 and 19-20, in order to keep the format of references consistent. Please kindly check.} {\color{blue} Thomas: good, I think the bibliography is perfect now. I have added references to a personal communication of Piotr Bizon, to an article of Bizon and an article of Shatah, and exchanged the references BMW and BFM to respect the alphabetical order. }}

\bibitem{DonningerSchorkhuber16} R. Donninger and B. Sch\"orkhuber, {\it On blowup in supercritical wave equations}, Commun. Math. Phys., \textbf{346} (2016), no. 3, 907-943.

\bibitem{DonningerSchorkhuber17} R. Donninger and B Sch\"orkhuber, {\it Stable blowup for wave equations in odd space dimensions}, Ann. Inst. Henri Poincar\'e, Anal. Non Lin\'eaire, \textbf{34} (2017), no. 5, 1181-1213.

\bibitem{G} R. Glassey, {\it Existence in the large for $\square u=F(u)$ in two space dimensions}, Math. Z., \textbf{178} (1981), no. 2, 233-261.

\bibitem{GMS} I. Glogi\'{c}, M. Maliborski and B. Sch\"{o}rkhuber, {\it Threshold for blowup for the supercritical cubic wave equation}, preprint, arXiv: 1905.13739.

\bibitem{GS} I. Glogi\'{c} and B. Sch\"{o}rkhuber, {\it Co-dimension one stable blowup for the super-critical cubic wave equation}, preprint, arXiv:1810.07681.

\bibitem{H} P. Hartman, {\it Ordinary Differential Equations}, New York: Birkh\"{a}user Boston, 1982.

\bibitem{J} F. John, {\it Blow-up of solutions of nonlinear wave equations in three space dimensions}, Manuscripta Math., \textbf{28} (1979), no. 1-3, 235-268.

\bibitem{JL} D. D. Joseph and T. S. Lundgren, {\it Quasilinear Dirichlet problems driven by positive sources}, Arch. Rational Mech. Anal., \textbf{49} (1973), no. 4, 241-269.



\bibitem{KW} O. Kavian and F. B. Weissler, {\it Finite energy self-similar solutions of a nonlinear wave equation}, Commun. PDE, \textbf{15} (1990), no. 10, 1381-1420.

\bibitem{K} J. Keller, {\it On solutions of nonlinear wave equations}, Comm. Pure Appl. Math., \textbf{10} (1957), no. 4, 523-530.


\bibitem{KS} J. Krieger and W. Schlag, {\it Large global solutions for energy supercritical nonlinear wave equations on $\mathbb{R}^{3+1}$}, J. Anal. Math., \textbf{133} (2017), no. 1, 91-131.

\bibitem{R} R. Kycia, {\it On self-similar solutions of semilinear wave equations in higher space dimensions}, Appl. Math. Comput., \textbf{217} (2011), 9451-9466.

\bibitem{L} L. A. Lepin, {\it Countable spectrum of the eigenfunctions of the nonlinear heat equation with distributed parameters}, Differential Equations, \textbf{24} (1988), 799-805.

\bibitem{MerleZaag08} F. Merle and H. Zaag, {\it Openness of the set of non-characteristic points and regularity of the blow-up curve for the 1 {D} semilinear wave equation}, Comm. Math. Phys., \textbf{282} (2008), no. 1, 55-86.

\bibitem{MerleZaag12} F. Merle and H. Zaag, {\it Existence and classification of characteristic points at blow-up for a semilinear wave equation in one space dimension}, Amer. J. Math., \textbf{134} (2012), no. 3, 581-648.

\bibitem{Shatah} J. Shatah, {\it Weak solutions and development of singularities of the $SU(2)$ $\sigma$-model}  Comm. Pure Appl.
Math., \textbf{41} (1988), no. 4, 459-469.
\end{thebibliography}

\end{document}